\documentclass{aims} 
\usepackage{amsmath}
\usepackage{amsthm}
\usepackage{amssymb}
\usepackage{mathrsfs}
\usepackage{paralist}
\usepackage{subcaption}
\usepackage[dvipsnames]{xcolor}
\usepackage[misc]{ifsym}
\usepackage{epsfig} 
\usepackage{epstopdf} 
\usepackage[colorlinks=true]{hyperref}
\hypersetup{urlcolor=blue, citecolor=red}
\allowdisplaybreaks

\def\currentvolume{X}
 \def\currentissue{X}
  \def\currentyear{200X}
   \def\currentmonth{XX}
    \def\ppages{X-XX}
     \def\DOI{10.3934/xx.xxxxxxx}

\newtheorem{theorem}{Theorem}[section]

\newtheorem{lemma}[theorem]{Lemma}
\newtheorem{proposition}[theorem]{Proposition}

\theoremstyle{definition}
\newtheorem{definition}[theorem]{Definition}
\newtheorem{remark}[theorem]{Remark}
\newtheorem{assumption}[theorem]{Assumption}

\title[Deep learning enhanced cost-aware multi-fidelity UQ]
{Deep learning enhanced cost-aware multi-fidelity uncertainty quantification of a computational model for radiotherapy} 

\author[Piermario Vitullo and Nicola Rares Franco and Paolo Zunino]{}

\subjclass{Primary: 65L60, 68T07, 65N22, 65M22, 65Z05.}
\keywords{multi-fidelity,
uncertainty quantification,
reduced order modeling,
deep learning,
neural networks,
radiotherapy.}

\thanks{$^*$Corresponding author: Paolo Zunino}

\begin{document}
\maketitle

\centerline{\scshape
Piermario Vitullo$^{1}$, Nicola Rares Franco$^{1}$
and Paolo Zunino$^{{\href{mailto:paolo.zunino@polimi.it}{\textrm{\Letter}}}*1}$}

\medskip

{\footnotesize
 \centerline{$^1$MOX, Department of Mathematics, Politecnico di Milano, Italy}
} 



\bigskip

 \centerline{(Communicated by Handling Editor)}


\begin{abstract}

Forward uncertainty quantification (UQ) for partial differential equations is a many-query task that requires a significant number of model evaluations. The objective of this work is to mitigate the computational cost of UQ for a 3D-1D multiscale computational model of microcirculation. To this purpose, we present a deep learning enhanced multi-fidelity Monte Carlo (DL-MFMC) method that integrates the information of a multiscale full-order model (FOM) with that coming from a deep learning enhanced non-intrusive projection-based reduced order model (ROM). The latter is constructed by leveraging on proper orthogonal decomposition (POD) and mesh-informed neural networks (previously developed by the authors and co-workers), integrating diverse architectures that approximate POD coefficients while introducing fine-scale corrections for the microstructures. The DL-MFMC approach provides a robust estimator of specific quantities of interest and their associated uncertainties, with optimal management of computational resources. In particular, the computational budget is efficiently divided between training and sampling, ensuring a reliable estimation process suitably exploiting the ROM speed-up. Here, we apply the DL-MFMC technique to accelerate the estimation of biophysical quantities regarding oxygen transfer and radiotherapy outcomes. Compared to classical Monte Carlo methods, the proposed approach shows remarkable speed-ups and a substantial reduction of the overall computational cost.

\end{abstract}


\def\eg{\textit{e.g.}\ }
\def\ie{\textit{i.e.}\ }
\def\argmin{\mathrm{argmin}}
\def\esssup{\mathrm{ess\,sup}}
\def\I{\boldsymbol{\mathrm{I}}}
\def\Tau{\mathcal{T}}
\def\domega{\partial \Omega}
\def\RR{\mathbb{R}}
\def\NN{\mathbb{N}}
\def\QQ{\mathbb{Q}}
\def\VV{\mathbb{V}}
\def\WW{\mathbb{W}}
\def\SS{\mathbb{S}}
\def\ann{\mathcal{N}}
\def\lnn{\mathcal{L}}
\def\minn{\mathcal{M}}
\def\wwmic{\boldsymbol{\widetilde{\theta}}_{g}}
\def\uu{\mathbf{u}_{\boldsymbol{\mu}}}
\def\qoi{Q_{\mathbf{u}}}
\def\qfom{Q_{\mathbf{u}}^{(FOM)}}
\def\qrom{Q_{\mathbf{u}}^{(ROM)}}
\def\para{\boldsymbol{\mu}}
\def\inlets{\boldsymbol{\eta}}
\def\dist{\mathbf{d}}
\newcommand{\nmax}{n_{\max}}
\def\nnmax{\left\lfloor \frac{p}{w_0} \right\rfloor}
\def\nnnmax{\left\lfloor p/w_0 \right\rfloor}
\newcommand{\bct}{{C}_t}
\newcommand{\bp}{{p}}
\def\uu{\mathbf{u}}
\def\MSE{\textnormal{MSE}}
\def\uph{u_{\para}^{\text{FOM}}}
\def\uuph{\mathbf{u}_{\para}^{\text{FOM}}}
\def\urom{\mathbf{u}_{\para}^{\text{ROM}}}
\def\uromm{\mathbf{u}^{\text{ROM}}}
\def\urommj{\mathbf{u}^{\text{ROM}_{j}}}
\def\ufom{\mathbf{u}^{\text{FOM}}}
\def\ifom{I_{\textnormal{FOM}}^{\gamma}}
\def\imfmc{I_{\textnormal{DL-MFMC}}^{\gamma}}
\def\survfraction{S_f}
\def\OER{\text{OER}}
\def\TCP{\text{TCP}}
\def\regularization{\xi}

\newcommand{\betareg}{\tau}
\newcommand{\alphareg}{\zeta}
\newcommand{\conflevel}{\gamma}
\newcommand{\alphaox}{\alpha_{t}} 
\newcommand{\alphaMC}{\lambda}
\newcommand{\estimator}{\hat{E}_n^{\textnormal{DL-MFMC}}}
\newcommand{\mse}{\textnormal{MSE}}
\newcommand{\MFMCestimator}{\hat{E}_{m_{0}, m_{1}}^{\textnormal{MFMC}}}
\newcommand{\DLMFMCestimator}{\hat{E}^{\textnormal{DL-MFMC}}}
\newcommand{\MFMCestimatorOpt}{\hat{E}_{m_{0}^*, m_{1}^*}^{\textnormal{MFMC}}}
\newcommand{\FOMestimator}{\hat{E}^{\textnormal{FOM}}}
\newcommand{\variance}{\textnormal{Var}}
\newcommand{\covariance}{\textnormal{Cov}}
\newcommand{\gcost}{g}
\newcommand{\train}{\textnormal{train}}
\newcommand{\nstar}{n_*}
\section{Introduction}
\label{sec:intro}

Uncertainty quantification (UQ) is a key aspect of computational modeling \cite{smith2013}. The challenge of UQ becomes even more pronounced in the context of life sciences, where the complexity of physical models is combined with the high uncertainties of the data \cite{Croci2021}. A particular case of this broad scenario, and the focus of this work, is the problem of microcirculation and oxygen transfer in biological tissues, here addressed using the multiscale model proposed in \cite{Possenti20213356}.
Quantifying how uncertainties propagate through these models, which are governed by complex partial differential equations (PDEs), requires extensive computational resources. In fact, to obtain robust and reliable estimates, numerous simulations are required, which means that the numerical solver has to be queried multiple times (a so-called \textit{many-query} scenario).
For this reason, traditional UQ methods that rely solely on high-fidelity simulations are often impractical, especially when it comes to multiscale models, where different scales need to be resolved accurately, and parametric data are high-dimensional \cite{villa2022}.

To address these issues, the state-of-the-art in UQ has been evolving rapidly, with multi-fidelity methods emerging as a promising approach to mitigate the computational burden \cite{peherstorfer2018survey}.
Simply put, multi-fidelity methods are based on the concept of combining high-fidelity models with one or more lower-fidelity models. The idea is that, being computationally cheaper but less accurate, lower-fidelity models can be used to inform and accelerate the computation performed by the high-fidelity model \cite{peherstorfer2018survey, ng2014multifidelity, guo2022multi, conti2022multi}. Furthermore, the applicability of multi-fidelity techniques has been further enhanced by recent advances in the reduced-order modeling literature, a research field devoted to the development of suitable surrogate models encompassing accuracy and efficiency \cite{benner2015,hesthaven2015certified,QuartManzNegri}. In particular, a major contribution along this direction has been provided by nonintrusive techniques based upon interpolation and regression algorithms \cite{alsayyari2019nonintrusive, amsallem2012nonlinear, audouze2013nonintrusive, barrault2004empirical, berzicnvs2020standardized, chen2018greedy, hesthaven2018non, gao2021non, FrancoROMPDE, wang2019non}.  In general, these approaches have gained traction in UQ due to its ability to significantly reduce computational time while maintaining acceptable levels of accuracy and robustness \cite{chen2017reduced, cicci2023reduced}.

Over the years, multi-fidelity methods have evolved from simple model hierarchies to more sophisticated techniques that intelligently balance the trade-off between computational cost and accuracy. One such technique is the multi-fidelity Monte Carlo (MFMC), which, complemented by a suitable strategy for optimal management of computational resources, has shown great promise in efficiently estimating the statistics of model outputs under uncertainty \cite{peherstorfer2016optimal}. This methodology, provides a comprehensive framework for combining an arbitrary number of surrogate models while also balancing offline and online costs. Scientifically speaking, its derivation is particularly innovative, as it moves away from the traditional reliance on error decay rates. Instead, it introduces an optimization problem that accounts for both errors and costs, acknowledging the fact that ROMs require a training phase in order to be operational. The paper demonstrates that, under certain mild conditions, the optimization problem admits a unique analytic solution, thus providing a practical way to manage the computational budget.

In this work, we contribute to the evolving landscape of MFMC methods by extending the ideas in \cite{peherstorfer2016optimal} to the case of deep learning-based ROMs, with a strong emphasis on our motivating application in relation to a multiscale computational model for microcirculation. More precisely, referring to the comprehensive model presented in \cite{Possenti2019b,Possenti20213356}, we address oxygen transfer from microvessels to interstitial tissue. This model provides a description of the oxygen field in the microvascular environment, which directly affects the performance of radiotherapy: see, e.g., \cite{Possenti2023.10.16.562646} and references therein.

Our approach integrates a full-order model (FOM) with a non-intrusive projection-based reduced order model (ROM) enhanced by deep learning, previously developed by the authors in \cite{Vitullo2024}. Simply put, the latter integrates proper orthogonal decomposition (POD) with mesh-informed neural networks \cite{FrancoMINN}, effectively relying on a closure modeling technique that resembles a separation of scales approach. This combination not only reduces the computational cost, but also retains the essential features captured by the FOM, spanning both local and global scales. We term our deep learning-enhanced MFMC approach as DL-MFMC. 

Building upon \cite{peherstorfer2016optimal}, we construct the DL-MFMC approach in order to account for two additional sources of computational cost: on the one hand, the training of the neural network architectures; on the other hand, the computational effort required to generate the input data, which, in microcirculation studies can be non-negligible. 
By applying our method to the estimation of statistics related to oxygen transfer and radiotherapy in microcirculation, we demonstrate its effectiveness in performing robust and reliable UQ analysis in a computationally efficient manner.
In this sense, our work not only contributes to the development of UQ multi-fidelity methods, but it also addresses a specific challenge in the context of microcirculation and radiation therapy.

The paper is organized as follows. First, in Section \ref{sec:dlmfmc} we set the notation and present the DL-MFMC approach in full generality, discussing the construction of our multi-fidelity estimator and the corresponding optimal management of computational resources. Then, in Section \ref{sec:microcirc} we dive into the details of our motivating application, presenting the FOM, the ROM, and the quantities of interest that we wish to estimate. Finally, we devote Section \ref{sec:results} to the numerical experiments and draw the corresponding conclusions in Section \ref{sec:conclusions}.
\section{A deep learning enhanced multi-fidelity Monte Carlo estimator}
\label{sec:dlmfmc}

In this section, we briefly review two different approaches to estimate the statistics of specific quantities of interest and provide suitable confidence intervals.

\subsection{Problem setup}
\label{sec:setup}
We are given a parametrized partial differential equation (PDE) of the form 
\begin{equation}\label{eq:problem}
    \begin{cases}
    L_{\para} u_{\para} = f_{\para} & \text{in} \ \Omega,
    \\
    B_{\para} u_{\para}= g_{\para} & \text{on} \ \domega,
\end{cases}
\end{equation}
where $L_{\para}$ denotes a (semilinear, second order) elliptic operator, whereas $B_{\para}$ is a boundary operator exemplifying Dirichlet, Neumann or Robin conditions applied to $u_{\para}$ on $\domega$, including problems with mixed-type conditions. Clearly, the solution $u=u_{\para}$ to \eqref{eq:problem} depends on the parameter vector $\para$. Here, we assume $\para$ to be taking values in some parameter space $\mathcal{P} \subset \mathbb{R}^d$, endowed with suitable probability distribution $\mathbb{P}$ modeling uncertainties.

At the continuous level, we assume that the solution to \eqref{eq:problem} can be sought within a given Hilbert space $V$. Our main goal is to quantify how uncertainties in the model parameters can propagate through the PDE, ultimately affecting certain outputs of interest. To this end, let us fix the notation and formally introduce a (nonlinear) functional
$$Q: V\to\mathbb{R},$$
representing the quantity of interest (QoI). Then, our objective is to come up with an efficient strategy for computing the average response, that is, $\mathbb{E}_{\para\sim\mathbb{P}}\left[Q(u_{\para})\right]$.
\\\\
As a first step, we assume that a suitable high-fidelity discretization of \eqref{eq:problem} is available (the so-called FOM). For simplicity, we shall assume that the latter consists of a Galerkin projection of \eqref{eq:problem} onto a Finite Element (FE) space $V_h\subset V$ of dimension $N_h=\mathrm{dim}(V_h)$. As in the continuous case, the FOM defines a map $\mathcal{P} \, \mapsto \, V_h$ mapping parameters onto FOM solutions $\para\mapsto\uph$.
From the discrete standpoint, the FOM is equivalent to a (large) system of algebraic equations of the form
\begin{equation}   
    \label{eq:algebraic}
    \mathbf{A}_{\para}\uuph = \mathbf{f}_{\para}
\end{equation}
where $\uuph \in \mathbb{R}^{N_h}$ is the vector of degrees of freedom in the FE approximation. With little abuse of notation, we shall write $Q(\uuph)$ to intend $Q(\uph)$, that is: we identify $Q$ with its discrete counterpart, so as to directly operate with vectors rather than functions.
\\\\
From a theoretical point of view, the output of interest, namely $\mathbb{E}_{\para\sim\mathbb{P}}\left[Q(\uuph)\right]$, can be estimated by classical Monte Carlo sampling. However, this approach would require solving \eqref{eq:algebraic} multiple times, resulting in a massive consumption of computational resources. For this reason, we resort to multi-fidelity strategies relying upon ROMs. Mathematically speaking, the ROM is a computational unit that acts as a suitable surrogate of the FOM, that is $$\text{ROM}:\;\para\mapsto\urom\quad\quad\text{with}\quad\quad\urom\approx\uuph.$$
Here, we assume that the ROM consists of several neural network architectures, all trained within a supervised learning framework. In other words, to be operational, the ROM requires a preliminary training phase, which is conducted on a selected collection of labeled FOM simulations, $\{(\para_{i},\ufom_{i})\}_{i=1}^{n}$, where $\ufom_{i}:=\ufom_{\para_{i}}.$ 

In general, the idea is to train a ROM and construct a multi-fidelity estimator of the QoI, named $\MFMCestimator$ and defined later on, by combining $m_{0}$ high-fidelity simulations, which are accurate but expensive, with $m_{1}$ low-fidelity simulations, computed via the ROM, which are cheaper to evaluate but less accurate. Clearly, the whole procedure must be carried out in a suitable way that ensures an actual reduction in the overall computational cost. To better appreciate this, let $w_{0}$ and $w_{1}$ be the computational times associated with the FOM and ROM simulations, respectively. Assume that the ROM is trained on $n$ high-fidelity samples, whereas the multi-fidelity estimator is constructed using $m_{0}$ evaluations of the FOM and $m_{1}$ evaluations of the ROM. Then, the overall computational cost is
$$nw_{0}+t(n)+m_{0}w_{0}+m_{1}w_{1},$$
where $t=t(n)$ is the training time associated with the ROM. In contrast, a classical Monte Carlo estimator $\FOMestimator{N}$, constructed using $N$ FOM simulations, entails a computational cost of $Nw_{0}.$ Assuming the case of unbiased estimators, the uncertainties in the two estimates can be computed as
$$\variance(\MFMCestimator)\quad\quad\text{and}\quad\quad\variance(\FOMestimator{N}).$$
It is then clear that, given a computational budget $p$, a multi-fidelity approach would only advantageous if for $p=nw_{0}+t(n)+m_{0}w_{0}+m_{1}w_{1}=Nw_{0}$, one has $\variance(\MFMCestimator)<\variance(\FOMestimator{N})$.
We note that the computational cost (generally measured in terms of floating-point operations) and the CPU time (measured in clock ticks or in seconds and highly dependent on the machine architecture) are considered here to be mutually proportional and used interchangeably.

Our purpose is to obtain a reduction of computational cost by proposing a suitable strategy for optimal management of the computational resources, together with an explicit multi-fidelity estimator that leverages on the deep learning nature of the ROM. Before doing so, however, it is worth recalling some of the basic ideas underlying Monte Carlo and multi-fidelity Monte Carlo methods, such as confidence intervals and mean-square-error metrics.


\subsection{Monte Carlo uncertainty quantification}
\label{sec:mc-fom}
We start by recalling the standard Monte Carlo estimator for the mean of a QoI. Given a computational budget $p$, the latter is computed using FOM simulations as
\begin{equation}
\FOMestimator{N}= \frac{1}{N} \sum_{i=1}^{N}Q(\ufom_{\para_{i}}),
\label{eq:MC-FOM}    
\end{equation}
where $N$ is the number of high-fidelity evaluations, $N=\nnnmax$, with $w_0$ being the cost of a single FOM evaluation. Here, we assume that the realizations of the input parameters $\para_{1},\dots\para_{N}$ are drawn independently and according to $\mathbb{P}$.

By definition, the FOM estimator is unbiased, in the sense that
\begin{equation}
    \label{eq:unbiased}
    \mathbb{E}\left[\FOMestimator{N}\right] = \mathbb{E}\left[Q\left(\ufom_{\para}\right)\right].
\end{equation}
Note that here the expected value on the right-hand-side is taken with respect to $\para\sim\mathbb{P}.$ Conversely, the one on the left is taken with respect to all possible outcomes of the sampling procedure, that is, $\para_{1},\dots,\para_{N}\sim\mathbb{P}\otimes\dots\otimes\mathbb{P}$.

As a direct consequence of \eqref{eq:unbiased}, the mean-squared-error (MSE) of the estimator 
\begin{equation}
    \label{eq:mse}
    \MSE(\FOMestimator{N}):=\mathbb{E}\left|\FOMestimator{N}-\mathbb{E}\left[Q(\ufom_{\para})\right]\right|^{2},
\end{equation}
is equal to the variance of the estimator. Due to independence, the latter is
\begin{equation}
\label{eq:variance}
\variance(\FOMestimator{N}) = \variance\left(\frac{1}{N} \sum_{i=1}^{N} Q(\ufom_{\para_{i}})\right) = \frac{1}{N^2}\sum_{i=1}^{N} 
                      \sigma_0^2 = \frac{\sigma_0^2}{N},
\end{equation}
where $\sigma_{0}^{2}:=\variance\;Q(\ufom_{\para})$ is the variance of the QoI, taken with respect to $\para\sim\mathbb{P}.$ In practice, $\sigma_{0}^{2}$ is usually estimated as 
\begin{equation}
\label{eq:estimatedvariance}
\widehat{\sigma}_0^2: = \frac{1}{N-1}\sum_{i=1}^{N}\left(\FOMestimator{N}-Q(\ufom_{\para_{i}})\right)^{2}.
\end{equation}
Using classical results from the theory of statistical estimators, these considerations can be exploited to construct confidence intervals, which are naturally ways to quantify uncertainties of pointwise estimates, such as \eqref{eq:MC-FOM}. We report a precise definition below. 

\begin{definition}[{\bf MC-FOM confidence interval}]
\label{def:ic-mc-fom}
Let $\para_{1},\dots,\para_{N}$ be $N$ random independent realizations of $\para\sim\mathbb{P}.$ Let $Q$ be a given quantity of interest. Fix a confidence level $\gamma\in(0,1)$. The MC-FOM confidence interval of level $\gamma$, $\ifom\subset\mathbb{R}$, is
\begin{equation}
    \label{eq:IC-MC-FOM}
\ifom:=\FOMestimator{N}\pm\mathrm{t}_{\frac{1-\gamma}{2},N}\sqrt{\frac{\widehat{\sigma}_0^2}{N}},\end{equation}
where $\FOMestimator{N}$ and $\widehat{\sigma}_0^2$ are as in Eq. \eqref{eq:MC-FOM} and \eqref{eq:estimatedvariance}, respectively. Here, $\mathrm{t}_{\alpha,q}$ denotes the quantile of the level $1-\alpha$ of a $t$-student distribution with $q$ degrees of freedom.
\end{definition}
The MC-FOM confidence interval provides a better estimate of the QoI, as it enriches the estimate in \eqref{eq:MC-FOM} with a quantification of uncertainties. The confidence level, $\gamma$, is related to how \textit{conservative} we want our estimate to be: the higher $\gamma$, the larger the interval (since we want to be \textit{more confident} about the fact that the interval captures the actual ground truth $\mathbb{E}[Q(\ufom_{\para})]$). More precisely, the formula in \eqref{eq:IC-MC-FOM} is constructed in such a way that
$$\text{Prob}\left(\mathbb{E}[Q(\ufom_{\para})]\in \ifom\right)\approx \gamma,$$
where $\text{Prob}=\mathbb{P}\otimes\dots\otimes\mathbb{P}$ is the joint probability distribution of the random sample, encoding the stochasticity of the confidence interval.

For a fixed confidence level, \eqref{eq:IC-MC-FOM} clearly shows that the uncertainty in the estimate decreases as a function of the FOM samples $N$. However, the decay is fairly slow, $\sim N^{-1/2}$. Consequently, a robust estimate may require a large number of FOM simulations. If the computational budget $p$ is limited and the cost of a single simulation $w_{0}$ is high, this approach may not be feasible.

\subsection{Multi-fidelity Monte Carlo estimator}
\label{sec:mfmc}

The driving idea behind multi-fidelity Monte Carlo (MFMC) is reduce the uncertainties in the final estimate by integrating Eq. \eqref{eq:MC-FOM} with some additional information, correlated with the QoI, but cheaper to compute. Here, this is achieved by relying on the ROM. For now, let us assume that the ROM is already available and fully operational (no training required).
Following our notation in Section \ref{sec:setup}, let $w_{0}$ be the computational time of a FOM simulation, and let $w_{1}$ be that of a ROM simulation. Let $$\para_{1},\dots,\para_{m_{0}},\;\para_{m_{0}+1},\dots,\para_{m_{1}},$$ be $m_{1}$ independent realization of the input parameter $\para\sim\mathbb{P},$ where $m_{0}<m_{1}.$ 
The MFMC approach is based on the following (unbiased) estimator, 
\begin{equation}
\label{eq:MFMC} 
\MFMCestimator := \FOMestimator{m_{0}} + \lambda \left( \hat{E}_{m_1}^{\text{ROM}} - \hat{E}_{m_0}^{\text{ROM}}  \right),   
\end{equation}
where $\FOMestimator{m_{0}}$ is as in Equation \eqref{eq:MC-FOM}, whereas
$\hat{E}_{m_k}^{\text{ROM}}:= m_{k}^{-1} \sum_{i=1}^{m_{k}}Q(\uromm_{\para_{i}})
$
is the ROM counterpart of the FOM estimator. Here, $\lambda>0$ is a suitable coupling parameter that regulates the impact of the ROM correction over the FOM estimate.

Note that, while we sampled a total of $m_{1}$ inputs, only $m_{0}$ of those were elaborated by the FOM (and the ROM). Instead, the remaining $m_{1}-m_{0}$ were only processed by the ROM. Thus, the overall computational cost of this procedure is $m_{0}w_{0}+m_{1}w_{1}$. From a UQ perspective, we also note that
\begin{equation}
    \label{eq:mse-mfmc}
    \MSE\left(\MFMCestimator\right) = \frac{\sigma_0^2}{m_0} + \left( \frac{1}{m_0} - \frac{1}{m_1} \right) (\lambda^2 \sigma_1^2 - 2 \lambda \rho \sigma_1 \sigma_0)
\end{equation}
where $$\sigma_{1}:=\variance\left(Q(\urom)\right)\quad\text{and}\quad\rho=\frac{\covariance\left(Q(\ufom), Q(\urom)\right)}{\sigma_{0}\sigma_{1}}.$$
In other words, the uncertainty associated to \eqref{eq:MFMC} depends on $m_{0},m_{1},\lambda,\sigma_{0},\sigma_{1}$ and $\rho.$ In general, given a computational budget $p$, the idea of the MFMC approach is to choose $m_{0},m_{1}$ and $\lambda$ by minimizing the uncertainty in \eqref{eq:mse-mfmc}, subject to the constraints $0\le m_{0}\le m_{1}$ and $m_{0}w_{0}+m_{1}w_{1}=p.$

In practice, this results in an optimization problem, whose optimal solution, $m_{0}^{*},m_{1}^{*},\lambda^{*}$, is known in closed form (at least under suitable mild assumptions: cf. \cite{peherstorfer2016optimal}). Notably, the optimal coupling parameter turns out to be $\lambda^{*}=\rho\sigma_{0}/\sigma_{1}$, meaning that $\lambda^{*}\neq0$ whenever $\rho\neq0.$ Since the FOM and the ROM are typically correlated, this suggests favoring MFMC over MC-FOM. Indeed, if the ROM is sufficiently cheap to evaluate, it can be shown that, compared to a naive MC-FOM estimator using all computational resources for FOM simulations, i.e. with $N=\nnnmax$, the optimal MFMC estimator entails a lower MSE (and thus, a lower uncertainty). For the interested reader, we refer to \cite[Corollary 3.5]{peherstorfer2016optimal}. Here, it is only worth mentioning that the MSE of the optimal MFMC estimator reads
\begin{equation}
    \label{eq:optmse}
        \MSE(\MFMCestimatorOpt) = \frac{\sigma_0^2}{p} \left( \sqrt{w_0 (1- \rho^2)} + \sqrt{w_1 \rho^2} \right)^2
\end{equation}
This property will come in handy for our construction in Section \ref{sec:dlmfmc}.
\\\\
While fascinating, this analysis has a major limitation. In fact, by taking the ROM for granted, it completely ignores the \textit{offline} cost regarding the construction and the training of the ROM. In fact, in order to be operational, ROMs typically necessitate of a preliminary training phase, in which they learn to approximate FOM simulations. Thus, part of the computational budget must be devoted to the generation of a suitable training set. 

In the literature, this fact was first acknowledged by Farcaș et al. in \cite{farcas2023context}. There, the authors noted that, if the ROM is trained on $n$ FOM simulations:
\begin{itemize}
    \item[i)] the remaining budget for a MFMC routine is $p-nw_{0};$
    \item[ii)] the quality of the ROM, especially in terms of QoI correlations, may depend on $n$.
\end{itemize}
Consequently, an optimal management of the computational resources requires a careful understanding of how $n$ enters into the equations. In \cite{farcas2023context}, by leveraging on correlation bounds, the authors characterize the optimal training size $n=\nstar$ as the solution to a suitable optimization problem involving the MSE of the MFMC estimator.

Here, instead of diving into details of the approach proposed by Farcaș et al., we directly present our adaptation to Deep Learning based ROMs. Compared to \cite{farcas2023context}, our analysis exhibits two major differences. First of all, aside from the generation of the training data, we also include the training time $t=t(n)$ within the offline costs, acknowledging the fact that some of the computational resources must be devoted to the actual training of the neural network architectures. Secondly, we account for the fact that, in complex applications, even the sampling of the input parameters may be computationally demanding. Finally, since Deep Learning based ROMs are extremely efficient, we set $w_{1}\equiv0$: that is, we assume ROM evaluations to have a negligible computational cost.

\subsection{A deep learning enhanced multi-fidelity estimator}
We are now ready to present our deep multi-fidelity estimator (DL-MFMC), specifically tailored for non-intrusive Deep Learning based ROMs. Similarly to the MFMC approach, our objective is to obtain a reduction of the uncertainties while also providing an optimal management of the computational resources. We articulate our presentation into three steps: i) definition of the DL-MFMC estimator, ii) formulation of the optimization problem, and iii) implementation of the optimal policy and construction of confidence intervals. 



\subsubsection{Notation}

As for the previous Sections, we assume to have access to a FOM, capable of producing high-quality simulations. In the case of complex multiscale systems, such as the microcirculation, oxygen transfer and radiotherapy model, a FOM simulation consists of two steps: the sampling of the parametric scenario and the consequent evaluation of the FOM solver. In general, the first step may entail a nonnegligible computational cost. For example, in the case of oxygen transfer models, the sampling step may be concerned with the computational synthesis of anatomically realistic vascular networks, encoded as a metric graph. For this reason, we model the cost of a single FOM simulation as
$$\gcost+w_{0},$$
where $\gcost$ and $w_{0}$ are the \textit{generation} and the \textit{evaluation} costs, respectively. The first term, $\gcost$, is intrinsic to the complexity of the problem, while the second one, $w_{0}$, is specific of the FOM.

Similarly to Section \ref{sec:mfmc}, we also assume that a suitable ROM technique has been chosen. For example, it may consist of a collection of multiple neural network architectures interacting in a suitable way, such as in the POD-NN \cite{hesthaven2018non}, DL-ROM \cite{FrancoMINN} and POD-MINN approaches \cite{Vitullo2024}. 
As we already mentioned, to be operational, the ROM must be trained on a collection of FOM samples (the so-called \textit{offline} phase). We model the corresponding \textit{offline} cost as
$$n(\gcost+w_{0})+t(n),$$
where $n$ is the number of FOM simulations in the training set. The first term, $n(\gcost+w_{0})$, is the sampling time, which corresponds to the generation of the training data. The second term, instead, corresponds to the actual training of the neural network architectures, here modeled through some non-decreasing monotone map $t:\NN_0\to\RR^+$. Instead, after training, the computational cost of a ROM simulation is $\gcost+w_{1}$ (input generation plus ROM evaluation). Here, we set $w_{1}\equiv0$ to emphasize the fact that online evaluations of Deep Learning-based ROMs can be carried out at negligible cost.

\subsubsection{Optimal training costs management}
\label{sec:optimal-policy}
We now reformulate the optimization problem proposed by Farcaș et al. in \cite{farcas2023context}, by adapting it to our framework. Let $p\in\RR_+$ be a computational budget.
We note that if we were to train a ROM on $n$ FOM simulations, we would be left with a computational budget $p_{n}=p-(\gcost+w_{0})n-t(n).$ At the same time, however, we would also have access to a trained ROM, available for multi-fidelity estimation. Specifically, we could leverage the MFMC paradigm, outlined in Section \ref{sec:mfmc}, to optimally utilize the remaining resources. In practice, for a fixed $n$, this procedure would bring us to the following estimator
\begin{equation}
\label{eq:DL-MFMC-n} 
\DLMFMCestimator_n := \FOMestimator{m_{0}^*(n)} + \lambda^*(n) \left( \hat{E}_{m_1^*(n)}^{\text{ROM}} - \hat{E}_{m_0^*(n)}^{\text{ROM}}  \right),
\end{equation}
where $m_0^*(n)$, $m_1^*(n),$ and $\lambda^*(n)$ are the optimal policy associated the computational budget $p_n$. More precisely, 
$$\lambda^*(n)=\rho(n)\frac{\sigma_0}{\sigma_1(n)},$$
where $\rho=\rho(n)$ and $\sigma_1=\sigma_1(n)$ are the correlation between FOM and ROM QoIs, and the variance of ROM QoIs, respectively. Here, we allow both to depend on $n$, so as to account for the role played by the training procedure. Similarly,
\begin{multline*}
    m_0^*(n),\;m_1^*(n)=\\=\argmin_{m_{0},m_{1}}\;\;\left\{\frac{\sigma_0^2}{m_0} - \left( \frac{1}{m_0} - \frac{1}{m_1} \right)\rho^2(n) \sigma_0^2\quad\text{s.t.}\;\;
\begin{array}{l}
     0\le m_0\le m_1,  \\
     m_0w_0+m_1\gcost=p_n 
\end{array}\right\},
\end{multline*}
which is nothing but \eqref{eq:mse-mfmc} up to substituting $\lambda$ with $\lambda^*(n)$. Note that, differently from Section \ref{sec:mfmc}, the multi-fidelity sampling cost is now $m_0w_0+m_1g$, rather than $m_0w_0+m_1w_1$. In fact, the multi-fidelity routine entails: 
\begin{itemize}
    \item [i)] generating $m_1$ parameter instances $\mapsto$ cost: $m_1\gcost$;
    \item [ii)] evaluating the FOM on $m_0$ of such instances $\mapsto$ cost: $m_0w_0$;
    \item [iii)]  evaluating the ROM on all parameter instances, $\mapsto$ cost: negligible.
\end{itemize}  

\noindent By leveraging on the theory of MFMC estimators, one can then easily prove the following.
\begin{proposition}
    \label{prop:optimal-dl-mse-n}
    Let $p>0$ be a given budget and let $n\in\mathbb{N}_{+}$
    be an admissible training size such that $p-(g+w_0)n-t(n)>0.$ Assume that $w_0\rho^2(n)>\gcost(1-\rho^2(n))$. Then
    \begin{equation}
    \label{eq:optdlmse}
        \MSE(\DLMFMCestimator_n) = \frac{\sigma_0^2}{p-(g+w_{0})n-t(n)} \left( \sqrt{w_0 (1- \rho(n)^2)} + \sqrt{\gcost\rho(n)^2} \right)^2.
\end{equation}    
\end{proposition}
\begin{proof}
    As we noted previously, for a fixed training size $n$, and a corresponding trained ROM, constructing the DL-MFMC estimator is equivalent to constructing a MFMC estimator with budget $p_n=p-(\gcost+w_0)n-t(n)$ and cost function $m_0,m_1\mapsto m_0w_0+m_1\gcost$. Then, \eqref{eq:optdlmse} can be easily derived from \eqref{eq:optmse} up to substituting $p$ with $p_n$ and $w_1$ with $\gcost$. Here, the formula can be applied, since condition $w_0\rho^2(n)>\gcost(1-\rho^2(n))$ precisely translates into the efficiency condition required in the MFMC paradigm: see, e.g., \cite[Theorem 3.4]{peherstorfer2016optimal}.
\end{proof}

Since Eq. \eqref{eq:optdlmse} only depends on $n$, this suggests the possibility of finding an optimal sample size, $n=\nstar$, by minimizing \eqref{eq:optdlmse} with respect to $n.$ Indeed, this is what we are going to do. To this end, we start by characterizing the dependency of $t=t(n)$ and $\rho=\rho(n)$ on $n$.
We do this by relying on the following assumptions.

\begin{assumption}
\label{a:rho}
$\exists \alphareg,c_{1},c_{2}>0$ such that $1 - \rho^2(n) \leq c_1 n^{-\alphareg}+c_{2}$ for all $n\in\mathbb{N}$, $n\ge1.$
\end{assumption}

\begin{assumption}
$\exists c_{3},c_{4}>0$ such that $t(n) \leq c_{3} n+c_{4}$ for all $n\in\mathbb{N}$, $n\ge1.$
\label{a:trainingtime}
\end{assumption}

\begin{assumption} $\forall n\in\mathbb{N}_{+}$, one has $w_0\rho^2(n)>(1-\rho^2(n))\gcost.$
\label{a:condition}
\end{assumption}

The first assumption states that the quality of ROM QoIs increases for larger datasets. The bounding expression is similar to the one proposed in \cite{farcas2023context}, but also comes with an additional term, namely $c_{2}>0$, which accounts for potential limitations inherent in the ROM. In fact, even if provided with an infinite amount of data, certain ROMs might still be incapable of replicating FOM simulations in their entirety.

The second assumption, instead, states that the training time is linear in the sample size $n$. In practice, deep learning-based ROMs always satisfy this assumption. In fact, the training of a neural network model $\phi:\mathcal{X}\to\mathcal{Y}$ typically involves the minimization of a loss function of the form 
\begin{equation}
    \label{eq:abstract-loss}
    \mathscr{L}(\phi)=\sum_{i=1}^{n}\ell(\phi(x_{i}), y_{i}),
\end{equation}
where $\{(x_{i},y_{i})\}_{i=1}^{n}\subset\mathcal{X}\times\mathcal{Y}$ denote an abstract training set, whereas $\ell:\mathcal{Y}\times\mathcal{Y}\to[0,+\infty)$ is a suitable discrepancy measure. Clearly, up to fixing a total number of training epochs, minimizing \eqref{eq:abstract-loss} requires at most $O(n)$ operations: in fact, each term in the sum can be tackled separately, even when computing gradients. Here, the constant $c_{4}$ models a fixed cost, in relation, for example, to the initialization of the ROM.

Finally, the third assumption states that most of the cost lies in the evaluation of the FOM solver, rather than in the generation of the input data. In fact, if $\rho(n)\ge1/2$, the condition results in $w_0>\gcost$. In most applications, such as ours, this condition is typically satisfied.
\\\\
Under these assumptions, the MSE in Eq. \eqref{eq:optdlmse} can be bounded as follows.

\begin{lemma}
Let Assumptions \ref{a:rho} - \ref{a:condition} hold.
For all $n\in\mathbb{N}$, $n\ge1$, the $\mse$ of  $\DLMFMCestimator_n$ can be bounded as
\begin{equation}
    \label{eq:bound}
    \mse(\DLMFMCestimator_n) \leq \frac{2\sigma_0^2}{p - (\gcost+w_0) n - c_3 n - c_{4}} \left ( c_1 w_0 n^{-\alphareg} + c_{2}w_{0} + \gcost \right).
\end{equation}
\label{lm:estimatebound}
\end{lemma}
\begin{proof} By Assumptions \ref{a:rho}-\ref{a:condition}, we have
\begin{multline*}
   \mse(\estimator)  =  \frac{\sigma_0^2}{p - (\gcost + w_0) n - t(n)} \left( \sqrt{w_0 (1- \rho^2(n))} + \sqrt{\gcost  \rho^2(n)} \right)^2 \\
     \leq \frac{\sigma_0^2}{p - (\gcost + w_0) n - c_3 n-c_{4}} \left( \sqrt{c_1 w_0 n^{-\alphareg}+c_{2}w_{0}} + \sqrt{\gcost } \right)^2 \\
     \leq \frac{2\sigma_0^2}{p - (\gcost + w_0) n - c_3 n - c_{4}} \left ( c_1 w_0 n^{-\alphareg} + c_{2}w_{0} + \gcost \right),
\end{multline*}
where we exploited the fact that $\rho(n)^{2}\le1$, and $(a+b)^{2}\le2(a^{2}+b^{2})$ for all $a,b\in\mathbb{R}.$
\end{proof}

Now, the idea is choose the training size $n$ by minimizing the upper bound in Lemma \ref{lm:estimatebound} rather than the $\mse$ of the estimator itself. As we shall prove in a moment, this results in a minimization problem admitting a unique minimizer $\nstar$, consistent with the budget constraint. We formalize these considerations in Proposition \ref{prop:uniqueness}, right after the auxiliary Lemma \ref{lemma:convexity}.

\begin{lemma}
    \label{lemma:convexity}
    Let $A>0$ and $\zeta>0$. The functions
    $$\varphi_{1}(x):=\frac{1}{A-x},\quad\quad\quad\quad
    \varphi_{2}(x):=\frac{x^{-\zeta}}{A-x}$$
    are convex in $(0,A)\subset\mathbb{R}$. Specifically, $\varphi_{1}''(x)>0$ and $\varphi_{2}''(x)\ge0$ for all $x\in(0,A).$
\end{lemma}

\begin{proof}
    It is straightforward to see that
    $$\varphi''_{1}(x)=2(A-x)^{-3}$$
    for all $x\in(0,A)$. Thus, $\varphi_{1}''>0$ in $(0,A).$ As for $\varphi_{2}$, a direct computation shows that
    $$
    \varphi_{2}''(x)=\frac{(\alphareg+1)\alphareg(A-x)^{2}x^{-\zeta-2}-2x^{-\alphareg-1}\alphareg(A-x)+2x^{-\alphareg}}{(A-x)^{3}}.$$
    The numerator in the above can be rewritten and bounded as
    \begin{multline*}
    x^{-\alphareg-2}\left[(\alphareg+1)\alphareg(A-x)^{2}-2x\alphareg(A-x)+2x^{2}\right]
    \\\ge
    x^{-\alphareg-2}\left[\alphareg^{2}(A-x)^{2}-2x\alphareg(A-x)+2x^{2}\right]
    \\
    =x^{-\alphareg-2}\left[\alphareg(A-x)-x\right]^{2}.
    \end{multline*}
    The conclusion follows.
\end{proof}

\begin{proposition}[Existence and uniqueness of the global minimum]
\label{prop:uniqueness}
Let $p>0$ be a given budget. 
Fix any $w_0, \gcost,c_{1},c_{2},c_{3},c_{4},\alphareg>0$ and let $\nmax:=(p-c_{4})/(\gcost+w_{0}+c_{3}).$ 
Then, the following function is strictly convex in $(0,\nmax)$:
\begin{equation*}
    F: n \mapsto \frac{2\sigma_{0}^{2}}{p - (\gcost+w_0) n - c_3 n-c_{4}} \left( c_1 w_0 n^{-\alphareg} +c_{2}w_{0} + \gcost \right),
\end{equation*}
Furthermore, it admits a unique global minimum $\nstar$ within $(0, \nmax)$.
\end{proposition}
\begin{proof}
We note that
$$F(n)=\frac{2\sigma_{0}^{2}c_{1}w_{0}}{\gcost+w_{0}+c_{3}}\cdot\frac{n^{-\alphareg}}{\nmax-n}+\frac{2\sigma_{0}^{2}(c_{2}w_{0}+\gcost)}{\gcost+w_{0}+c_{3}}\cdot\frac{1}{\nmax-n}.$$
Thus, by Lemma \ref{lemma:convexity}, $F$ is the positive sum of two convex functions, one of which is strictly convex. Consequently, $F''>0$ in $(0,\nmax)$. Since $F$ goes to infinity at the boundaries, $F(n)\to+\infty$ for $n\to0$ and $n\to\nmax$, it follows that $F$ admits a unique global minimum within $(0,\nmax).$ 
\end{proof}
\;\\
Using Proposition \ref{prop:uniqueness}, we finally propose the following optimal management policy:
\begin{itemize}
    \item [1.] Fix a computational budget $p>0$;\vspace{0.2cm}
    \item [2.] Generate a preliminary collection of FOM samples $n_0$, by consuming a small portion of the budget, $n_0(w_0+\gcost)\ll p$;\vspace{0.2cm}
    \item [3.] Initialize and train the ROM for different sample sizes $n_{1}<\dots<n_{k}\le n_{0}$, as to obtain auxiliary records about training times, $t(n_{1}),\dots,t(n_{k})$, and QoI correlations, $\rho(n_{1}),\dots,\rho(n_{k})$;\vspace{0.2cm}
    \item [4.] Estimate the coefficients $c_{1},c_{2},c_{3},c_{4},\alphareg$ by leveraging on $\{n_{j},t(n_{j}),\rho(n_{j})\}_{j=1}^{k};$\vspace{0.2cm}
    \item [5.] Find the optimal sample size $\nstar$ by solving the minimization problem in Proposition \ref{prop:uniqueness}. Without loss of generality, we assume $\nstar>n_0$;\vspace{0.2cm}
    \item [6.] Augment the training set by generating $\Delta n = \nstar-n_0$ new FOM simulations. The remaining budget is now $p-(\gcost+w_0)\nstar$;\vspace{0.2cm}
    \item [7.] Train a final ROM surrogate by using all $\nstar$ FOM simulations, and compute the corresponding QoI correlation coefficient $\rho=\rho(\nstar).$ The remaining budget is now $p-(\gcost+w_0)\nstar-t(\nstar);$\vspace{0.2cm}
    \item [8.] Compute $m_{0}^*=m_0^*(\nstar)$, $m_{1}^*=m_1^*(\nstar)$ and $\lambda_*=\lambda_*(\nstar)$ according to the MFMC paradigm, that is,
    \begin{equation}
\begin{cases}
    \displaystyle
    m_0^* = \frac{p - (\gcost+w_0) \nstar - t(\nstar)}{w_0 + w_1 r}, \vspace{0.2cm}\\\vspace{0.2cm}
    m_1^* = r m_0^*,\\ \displaystyle
    \lambda_*=\frac{\rho(\nstar)\sigma_{0}}{\sigma_{1}},
\end{cases}
\end{equation}
where $r^{2} := (w_0\rho^2(\nstar))/(\gcost(1-\rho^2(\nstar));$\vspace{0.2cm}
    
    \item [9.] Construct the DL-MFMC estimator
    \begin{equation}
    \label{eq:dlmfmc-estimator}
    \DLMFMCestimator := \DLMFMCestimator_{\nstar},
    \end{equation}
    by sampling the required FOM-ROM simulations, thus exhausting the computational budget.  
\end{itemize}

\begin{remark}
    We stress that, in practice, the values of $m_0^*$, $m_1^*$ and $\lambda_*$ are
    approximated by leveraging on empirical estimates of $\rho(\nstar)$, $\sigma_0$ and $\sigma_{1}(\nstar)$: however, in order to keep the notation lighter, we choose not to make this distinction explicit.
    Operationally, the idea goes as follows. First, we exploit the training data to calculate a preliminary estimate of $\rho(\nstar)$, which we use to calculate the optimal sample sizes $m_0^*$, $m_1^*$. Then, following the MFMC paradigm, we compute $m_0^*$ new independent FOM simulations. On these simulations we derive the final estimates of $\rho(\nstar)$, $\sigma_0$, $\sigma_{1}(\nstar)$, and consequently, $\lambda_{*}.$
\end{remark}

\begin{remark}
    Step n.3 in the optimal policy pipeline concerns estimating the correlation coefficients $\rho(n_{1}),\dots,\rho(n_{k})$. One way to achieve this is to directly utilize the training data. For any $j\in{1,\dots, k}$, let $\text{ROM}_{j}$ denote the ROM in its $j$th training iteration. Let $\para_{1},\dots,\para_{n_{j}}$ be the input parameters observed during training. Then
\begin{equation}
\label{eq:corr_estimate}
\frac{\sum\limits_{i=1}^{n_{j}}\left[Q(\ufom_{\para_{i}})Q(\urommj_{\para_{i}})-\left(\frac{1}{n_{j}}\sum\limits_{b=1}^{n_{j}}Q(\ufom_{\para_{b}})\right)\left(\frac{1}{n_{j}}\sum\limits_{b=1}^{n_{j}}Q(\urommj_{\para_{b}})\right)\right]}{\sqrt{\sum\limits_{i=1}^{n_{j}}\left[Q(\ufom_{\para_{i}})-\left(\frac{1}{n_{j}}\sum\limits_{b=1}^{n_{j}}Q(\urommj_{\para_{b}})\right)\right]^{2}\sum\limits_{i=1}^{n_{j}}\left[Q(\urommj_{\para_{i}})-\left(\frac{1}{n_{j}}\sum\limits_{b=1}^{n_{j}}Q(\urommj_{\para_{b}})\right)\right]^{2}}}\end{equation}
can be a crude approximation of $\rho(n_{j}).$
However, \eqref{eq:corr_estimate} would likely result in biased estimate: indeed, $Q(\urommj_{\para_{1}}),\dots,Q(\urommj_{\para_{n_{j}}})$ are, in general, statistically correlated (notice, in fact, that $\text{ROM}_{j}$ itself depends on the whole training set). To account for this, a better approach is to rely on a different set of data, independent of the training set. Operationally, this can be achieved by picking $n_{k}$ such that $n_{k}<n_{0}$. In this way, the remaining $n_{0}-n_{k}$ observations can be used as a common "test set", shared within the $j$ trainings. In other words, in iteration $j$, the ROM is trained on $\para_{1},\dots,\para_{n_{j}}$, but formula \eqref{eq:corr_estimate} is evaluated on $\para_{n_{0}-n_{k}},\dots,\para_{n_{0}}$.  
\end{remark}

\subsubsection{DL-MFMC estimator and uncertainty quantification}

Given a computational budget $p>0$, the DL-MFMC estimator $\DLMFMCestimator$ can be constructed following the optimal management policy in Section \ref{sec:optimal-policy}. By combining FOM and ROM simulations, the latter provides an efficient and robust estimate of $\mathbb{E}_{\para\sim\mathbb{P}}[Q(\ufom_{\para})]$, characterized by a significant reduction of the uncertainties. In practice, this fact becomes apparent when considering confidence intervals, rather than crude pointwise estimates. 

\begin{definition}[{\bf DL-MFMC confidence interval}]
\label{def:ic-dl-mfmc}
Let $Q$ be a given quantity of interest. For a given computational budget $p>0$, let $\DLMFMCestimator$ be the DL-MFMC estimator, computed as in Section \ref{sec:dlmfmc}. Fix a confidence level $\gamma\in(0,1)$. The DL-MFMC confidence interval of level $\gamma$ is
\begin{equation}
    \label{eq:IC-DL-MFMC}
\imfmc:=\DLMFMCestimator\pm\mathrm{z}_{\frac{1-\gamma}{2}}\sqrt{\frac{\hat{\sigma}_0^2}{m_0^*(\nstar)} - \left( \frac{1}{m_0^*(\nstar)} - \frac{1}{m_1^*(\nstar)} \right)\hat{\rho}^2(\nstar) \hat{\sigma}_0^2},\end{equation}
where $\mathrm{z}_{\gamma}$ denotes the quantile of level $1-\gamma$ of the standard Gaussian distribution $\mathcal{N}(0,1)$.
\end{definition}

Similarly to the MC-FOM case, the DL-MFMC confidence interval is constructed such that
$$\text{Prob}\left(\mathbb{E}[Q(\ufom_{\para})]\in \imfmc\right)\approx \gamma.$$
Notice, however, that Definition \ref{def:ic-dl-mfmc} uses normal quantiles rather than $t$-student ones. In fact, while the latter appear naturally when considering uncorrelated Monte Carlo samples, they do not extend to multi-fidelity estimators (which, in contrast, leverage on correlated observations). For this reason, we rely on more general quantiles, derived from the Gaussian distribution. The heuristics behind this choice lies in the Central Limit Theorem, according to which the independent sum of identically distributed random variables is asymptotically normally distributed. Since 
$$\DLMFMCestimator = \FOMestimator{m_{0}^*(\nstar)} + \lambda_*(\nstar)\hat{E}_{m_1^*(\nstar)}^{\text{ROM}} -\lambda_*(\nstar) \hat{E}_{m_0^*(\nstar)}^{\text{ROM}},$$
the DL-MFMC estimator can be seen as the sum of three random variables, $$X=\FOMestimator{m_{0}^*(\nstar)},\;\;\;\;Y=\lambda_*(\nstar)\hat{E}_{m_1^*(\nstar)}^{\text{ROM}},\;\;\;\;Z=-\lambda_*(\nstar) \hat{E}_{m_0^*(\nstar)}^{\text{ROM}},$$ each of which is asymptotically normal. Then, since the sum of normal random variables is also normal, we can approximate the distribution of $\DLMFMCestimator$ as $$\mathcal{N}\Big(\mathbb{E}[X+Y+Z],\;\variance(X+Y+Z)\Big)=\mathcal{N}\Big(0,\;\MSE(\DLMFMCestimator)\Big),$$
thus motivating the formula in Definition \ref{def:ic-dl-mfmc}.
\section{A model for oxygen transport in microcirculation and radiotherapy}
\label{sec:microcirc}


In this Section and in the following, we present an application of our approach to a comprehensive mathematical model of microcirculation, with a specific focus on the intricate interplay between oxygen transport and its implications for radiotherapy. Initially, we introduce the biophysical model of reference along with its associated quantities of interest; then, we provide a succinct overview of the FOM and its corresponding surrogate model. We point out that these three stages modeling of the physical phenomenon, high-fidelity discretization, and model order reduction were previously undertaken by the authors and their collaborators in earlier works: the interested reader can refer to \cite{Possenti2019b}, \cite{Possenti20213356}, and \cite{Vitullo2024}, respectively.

Having established the groundwork, we then move to the actual application of the DL-MFMC approach, offering a comprehensive discussion in Section \ref{sec:results}.


\subsection{Oxygen transport in microcirculation: mechanistic model and quantities of interest}
\label{sec:qois}

To model oxygen transport, we rely on the model presented by Possenti et al. in \cite{Possenti2019b, Possenti20213356}, which encompasses blood flow, hematocrit transport coupled with interstitial flow, and oxygen transport in both blood and tissue through vascular-tissue exchange. 

The general model describes flow in two distinct domains: the tissue domain ($\Omega \subset \RR^3$ with $\dim(\Omega)=3$), where the unknowns encompass fluid pressure $p_t$, fluid velocity $\uu_t$ and oxygen concentration $C_t$; and the vascular domain ($\Lambda \subset \RR^3$ with $\dim(\Lambda)=1$), representing a metric graph that describes a network of connected one-dimensional channels. In this domain, unknowns involve blood pressure $p_v$, blood velocity $\uu_v$, and vascular oxygen concentration $C_v$. The model for oxygen transport uses velocity fields $\uu_v$ and $\uu_t$ to describe blood flow in the vascular network and plasma flow in tissue. The governing equations for the oxygen transfer model are as follows:

\begin{equation}
\label{eq:oxy}
\begin{cases}
\displaystyle
\nabla \cdot \left(-D_t \nabla C_t + \uu_t~ C_t\right) 
+ V_{max} ~ \frac{C_{t}}{C_{t} + \alphaox ~p_{m_{50}}} = \phi_{O_2} \, \delta_\Lambda \quad &\textrm{on $\Omega$}
\\[10pt]       
\displaystyle
\pi R^2 \frac{\partial}{\partial s} \left(
- D_v \frac{\partial C_v}{\partial s} + {v_v} ~ C_v 
+ {v_v} ~ k_1~H~  \frac{C_v^\gamma}{C_v^\gamma + k_2}\right)  
= -  \phi_{O_2} \quad &\textrm{on $\Lambda$}
\\[8pt]
\displaystyle
\phi_{O_2} = 2 \pi R~P_{O_2} (C_v - \bct) + (1-\sigma_{O_2})~\left(\frac{C_v + \bct}{2}\right)~\phi_v \quad &\textrm{on $\Lambda$}
\\[10pt]
\displaystyle
\phi_v=2 \pi R L_p\big((p_v - \bp_t) - \sigma(\pi_v-\pi_t)\big)
\\[6pt]
\displaystyle
C_v = C_{in}  \quad &\textrm{on}~\partial\Lambda_{\text{in}}
\\[6pt]
\displaystyle
- D_v \frac{\partial C_v}{\partial s} = 0  \quad &\textrm{on}~\partial\Lambda_{\text{out}}
\\[6pt]
\displaystyle
- D_t \nabla C_t\cdot\mathbf{n} = \betareg_{O_2}~(C_t - C_{0,t}) \quad &\textrm{on}~\partial\Omega.
\end{cases}
\end{equation}

Specifically, the first equation governs the oxygen distribution within the tissue, the second outlines how oxygen is transported through the bloodstream, and the third defines the transfer of oxygen between the two domains, $\Omega$ and $\Lambda$. In particular, the flux $\phi_{O_2}$ is derived under the assumption that the vascular wall acts as a semi-permeable membrane. Complementing this model is a set of boundary conditions detailed in the final three equations: at the vascular inlets $\partial\Lambda_{\text{in}},$ we prescribe the oxygen concentration; at the vascular endpoints $\partial\Lambda_{\text{out}},$ null diffusive flux is enforced; and for the boundary of the tissue domain $\partial\Omega,$ we simulate the presence of an adjacent tissue domain with boundary conductivity $\betareg_{O_2}$ and a concentration in the far field $C_{0,t}$.

The symbols $D_t, D_v, V_{max}, \alphaox, p_{m_{50}}, k_1, k_2, C_v^\gamma, P_{O_2}, \sigma_{O_2}, L_p, \sigma, \pi_v, \pi_t$ represent constants independent of the model solution. For a complete explanation of the physical significance of these variables, refer, for example, to \cite{Possenti20213356}.
\begin{figure}
\centering
\includegraphics[width=0.3\textheight]{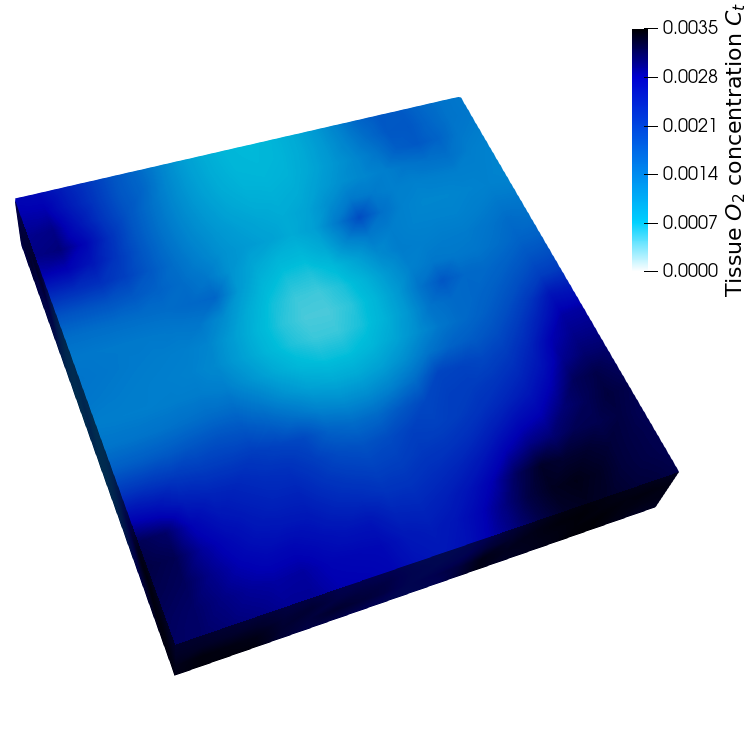}
\includegraphics[width=0.3\textheight]{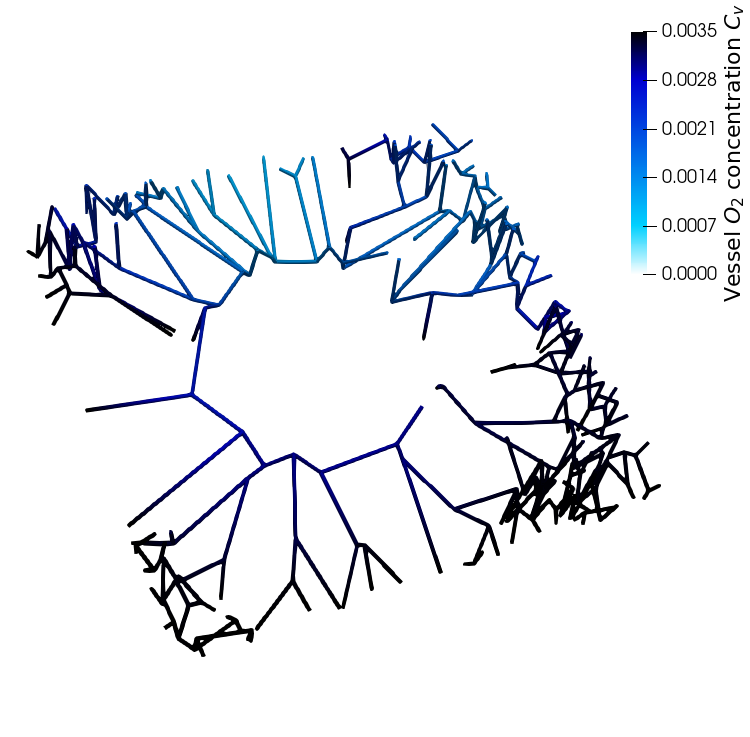}
\caption{The tissue oxygenation map, measured in $mL_O2/mL_B$, is visually depicted on the left panel through the FOM solution (light blue corresponds to low oxygen). On the right panel, we showc the 1D embedded vascular microstructure, which visibly impacts the oxygen map. Furthermore, oxygen concentration in the blood, indicated as $C_v$, is also reported.}
\label{fig:outputmicro}
\end{figure}

Our main interest is to perform a reliable analysis of certain quantities of interest, relevant for radiotherapy applications, when the topology of the vascular network $\Lambda$, and the values of the physical parameters $V_{max},C_{in},P_{O_{2}}$ are uncertain. The choice of these parameters is motivated by the sensitivity analysis study recently performed in \cite{VitulloSensitivity}. As we shall see in a moment, these quantities of interest can be expressed as certain functionals of the tissue oxygenation map $C_{t}$, herein measured in $mL_{O_2}/mL_B$. We refer to Figure \ref{fig:outputmicro} for a visual example. In this sense, 
$$u_{\para}=C_{t}\quad\text{and}\quad\para=[\Lambda,V_{max},C_{in},P_{O_{2}}],$$
according to our notation in Section \ref{sec:dlmfmc}. 
Here, we shall focus on three quantities of interest: average partial pressure, partial pressure variability, and tumor control probability. We detail them below.

%


\subsubsection*{Average oxygen partial pressure $\overline{pO_2}$}
To start, we consider the spatial average of the oxygen partial pressure, $\overline{pO_2}$, computed as
\begin{equation}
    \label{eq:avg_pO2}
    \overline{pO_2} = \frac{1}{\lVert\Omega\rVert} \int_{\Omega} \frac{C_t(\textbf{x})}{\alphaox},
\end{equation}
where $\alphaox$ is the oxygen solubility in the tissue. We note that $\overline{pO_2}=Q(C_{t})$, where $Q:L^{2}(\Omega)\to\mathbb{R}$ is linear and continuous with respect to the $L^{2}$-norm.


\subsubsection*{Oxygen partial pressure range $\Delta{pO_2}$}
Since the average partial pressure provides a global perspective on tissue perfusion, it is also interesting to consider different QoIs which are more sensitive to local fluctuations. To this end, we introduce a further QoI, $\Delta pO_2$, measuring the difference between the maximum and the minimum partial pressure. We call the latter \textit{oxygen partial pressure range}. Operationally, this is computed as
\begin{equation}
    \label{eq:delta_pO2}
    \Delta pO_2 = \frac{\max_{\substack{\textbf{x}\in \Omega}} C_t(\textbf{x}) - \min_{\substack{\textbf{x}\in \Omega}} C_t(\textbf{x})}{\alphaox}.
\end{equation}
From a clinical point of view, $\Delta pO_2$ quantifies the abundance of local hypoxic effects in the tissue. We note that, unlike the average partial pressure, this QoI is non-smooth and not even defined on $L^{2}(\Omega)$ as a whole. However, it can be regarded as continuous functional with respect to the $L^{\infty}$-norm, that is: $\Delta pO_2=Q(C_{t})$ with $Q:L^{\infty}(\Omega)\subset L^{2}(\Omega)\to\mathbb{R}$.

\subsubsection*{Tumor Control Probability (TCP)}
As a final example, we consider a radiotherapy-related QoI regarding the probability of tumor eradication (Tumor Control Probability, TCP for short). To this end, we model radiotherapy treatment using the linear-quadratic model, which is the most widely used radiological model. The latter is based on two different parameters that describe the radiosensitivity of cells or tissue. The first parameter $\alpha$ describes the lethal damage resulting from a single hit, while $\beta$ is related to multiple hits, namely the interaction of multiple radiation tracks \cite{McMahon2019}.
Combining these two parameters and the dose administered ($D$), we model the surviving fraction $\survfraction$ as a spatially dependent map, adopting the model proposed by Tinganelli et al. in \cite{Tinganelli2015} that accounts for the effect of oxygen on radiotherapy:
\begin{equation}
    \label{eq:LQ_TIN}
    \survfraction(D,pO_2,\mathbf{x}) = \exp\left( - \alpha\frac{D}{\OER(pO_2(\mathbf{x}),\;LET)} ~ -  \beta \left( \frac{D}{\OER(pO_2(\mathbf{x}),\;LET)}\right)^2 \right),
\end{equation}
where $pO_2=C_{t}/\alpha_{t}$ is the oxygen partial pressure, whereas the Oxygen Enhancement Ratio (\OER) is a suitable transformation operating a change of scale: 
\begin{equation*}
    \displaystyle
     \OER(0, LET) ~ =  \frac{LET^\delta + M ~ a}{a + LET^\delta},
\end{equation*}  
\begin{equation*}
     \OER(\mathbf{x}, pO_2, LET) ~ =  \frac{b ~  \OER(0,\;LET) + pO_2(\mathbf{x})}{b + pO_2(\mathbf{x})},
\end{equation*}
\begin{equation*}
     \displaystyle
     D_{\OER}(\mathbf{x},pO_2,LET) ~ =  \frac{D}{\OER(pO_2(\mathbf{x}),\;LET)}.
\end{equation*}
Here, $M$, $a$, $\delta$, $b>0$ are the model parameters fitted to the experimental data, $LET$ is the linear energy transfer of ionizing radiation, and $D_{\OER}(\mathbf{x}, pO_2, LET)$ is the dose corrected for the oxygen effect, to be included in the LQ model.
The admissible values of $\survfraction$ range from 0 to 1, representing the fraction of cells that survived treatment with the specified dose $D$, assumed as a constant.

Based on that, we define the Tumor Control Probability (TCP), which describes the probability of successful treatment.
It is based on the number of cells that survive treatment $\survfraction$ and the distribution of the initial number of clonogenic cells $N=N(\mathbf{x})$. For a fixed radiation dose $D$, 
and under the assumption of unicellular independence, following the approach proposed by Strigari et al. \cite{Strigari2018}, the TCP can be expressed as
\begin{equation}
    \label{eq:TCP}
    \TCP = \exp\left(-\int_{\Omega}N(\mathbf{x})S_{f}(D,C_{t}/\alpha_{t},\mathbf{x})d\mathbf{x}\right).
\end{equation}
We mention that although the definition of the TCP is far more involved compared to the one of the average partial pressure, the TCP can be realized as a Lipschitz continuous nonlinear functional of the oxygen concentration $C_{t}$, with respect to the $L^{2}$-norm. More precisely, TCP = $Q(C_{t})$ for some smooth operator $$Q:L^{2}_{+}(\Omega)\to\mathbb{R},$$
where $L^{2}_{+}(\Omega):=\{u\in L^{2}(\Omega)\;\text{s.t.}\;u\ge0\}$. We leave this simple verification to the reader.


\subsection{Description of the FOM: a high-fidelity model leveraging on the finite element method}
\label{subsec:fom}
As a high-fidelity model, we consider a FOM based off a Finite Element discretization of problem \eqref{eq:oxy}.
Following the approach outlined in \cite{Possenti2019b}, we achieve this by discretizing both the interstitial domain $\Omega$, representing a 3D slab with dimensions of $1,mm$ in edge length and $0.15,mm$ in thickness, and the embedded 1D vascular network $\Lambda$.

For the interstitial domain, we utilize a structured mesh composed of tetrahedral elements arranged in a $20\times20\times3$ grid, resulting in $N_h=1764$ degrees of freedom. Over this mesh, we define a corresponding space of piecewise linear continuous Lagrangian finite elements, denoted as $V_{t,h}=X_{h}^{1}(\Omega)$.
To discretize the vascular network, instead, we subdivide each vascular branch $\Lambda_i$ into multiple linear segments. Subsequently, we associate each branch with a corresponding finite element space $V_{v,h}^{i}=X_{h}^{1}(\Lambda_{i})$, comprising piecewise linear continuous Lagrangian elements.
With this setup, we instruct the FOM solver to produce a numerical solution to \eqref{eq:oxy} as
$$(C_t, C_v)\in V_{t,h}\times V_{v,h},$$
where $V_{v,h}=\left(\bigcup_{i=1}^{N_b}V_{v,h}^i \right) \bigcap \mathcal{C}^0(\Lambda)$.

Notice, however, that in order to assemble (and solve) the oxygen transfer model, one first needs to solve the underlying fluid flow problem, which takes place in the vascular microenvironment \cite{Possenti2019b,Possenti20213356}. Figure \ref{fig:layoutFOM} illustrates the sequential calculation of velocity, pressure, and discharge hematocrit in both tissue and vascular networks using the finite element method. We refer to \cite{Possenti20213356} for more details on the derivation and validation of the model.
\begin{figure}[htb]
\centering
\includegraphics[width=\textwidth]{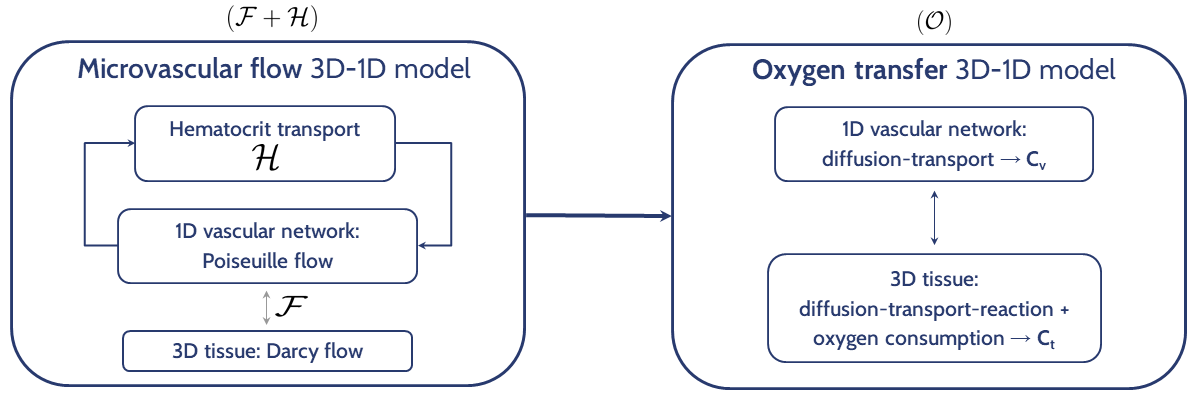}
\caption{General layout of the full order model for the whole vascular microenvironment.}
\label{fig:layoutFOM}
\end{figure}
\\\\
As we already mentioned, our main interest is to characterize the behavior of certain quantities of interest, related to $C_{t}$, under suitable uncertainties in the model parameters. In particular, following an earlier sensitivity analysis performed by Vitullo et al. \cite{VitulloSensitivity}, we focus our attention on the role played by $V_{max},$ $C_{v,in}$, $P_{O_2}$ and $\Lambda$.

In order to explore different scenarios, we sample the model parameters as follows. For the physical parameters —specifically $V_{max}$, $C_{v,in}$, and $P_{O_2}$— we randomly select values within the physiological range of variation in a uniform manner (see Table \ref{tab:parameters}). For the vascular network, instead, we utilize a random generator implementing a biomimetic algorithm that emulates the natural process of new blood vessel formation.
The output of this algorithm is driven by the value of two geometrical parameters: the vascular surface area per unit volume of tissue, denoted as $S/V$, and the percentage of seeds for angiogenesis. Essentially, the former represents the overall density of the vascular network, affecting the number and spacing of blood vessels within the tissue, while the latter governs the distribution of initial point seeds, serving as starting points for the algorithm to initiate new blood vessel growth. Together, these two parameters allow the generation of complex vascular networks, with different configurations of the blood vessels and extravascular regions of varying dimensions: see, e.g., Figure \ref{fig:networks}. Similarly to the physical parameters, the values of the geometrical parameters are sampled randomly within the corresponding range of variation.

\begin{figure}[htb]
    \centering
    \begin{subfigure}{0.48\textwidth}
    \includegraphics[scale=0.25]{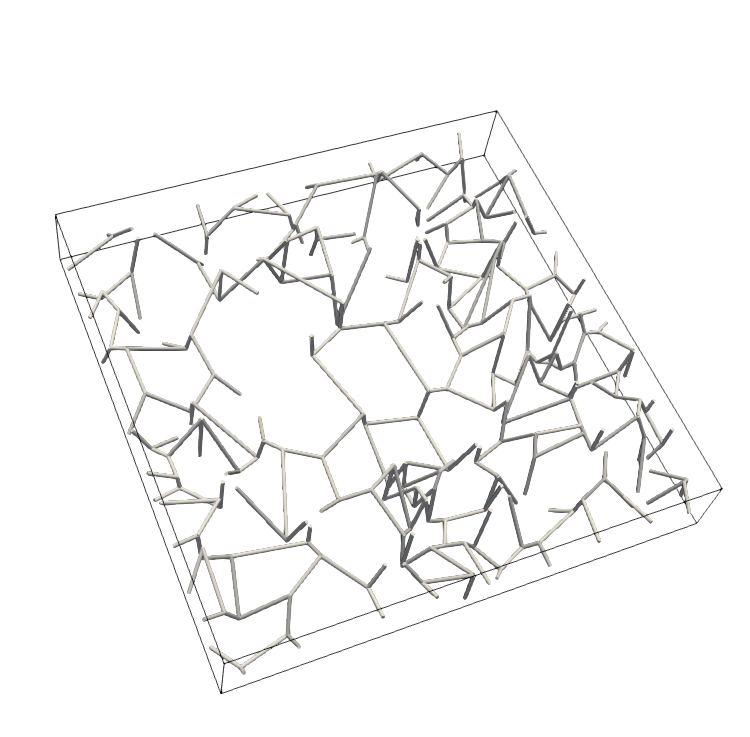}    
    \caption{Low vascularization, homogeneous distribution}
    \end{subfigure}
    \begin{subfigure}{0.48\textwidth}
    \includegraphics[scale=0.25]{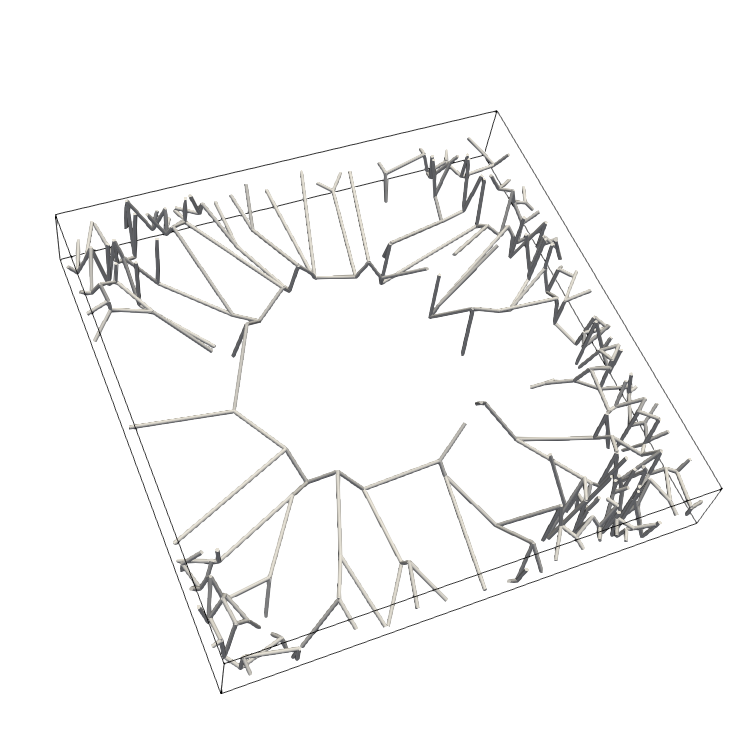}
    \caption{Low vascularization, high extravascular distance}
    \end{subfigure}
    \\
    \begin{subfigure}{0.48\textwidth}
    \includegraphics[scale=0.25]{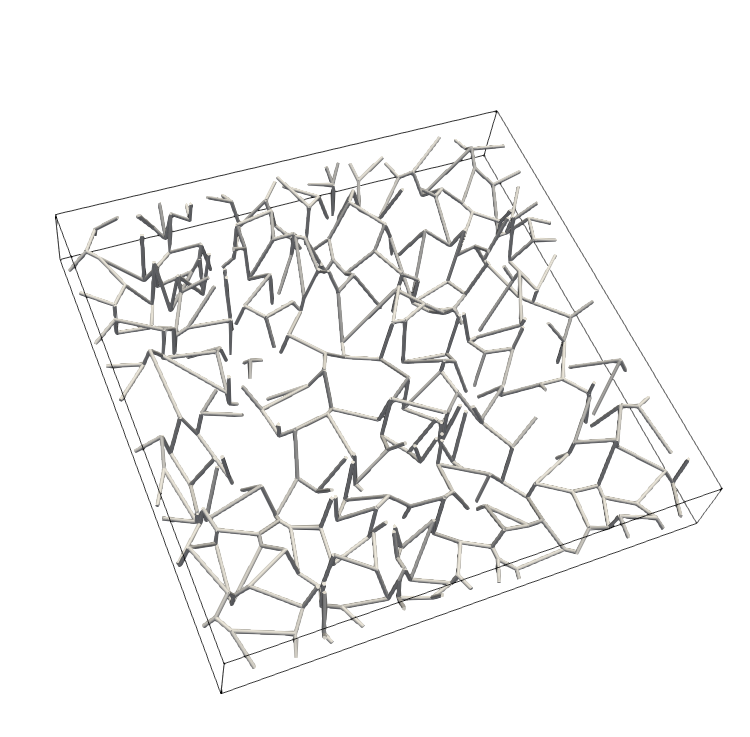}
    \caption{High vascularization, homogeneous distribution}
    \end{subfigure}
    \begin{subfigure}{0.48\textwidth}
    \includegraphics[scale=0.25]{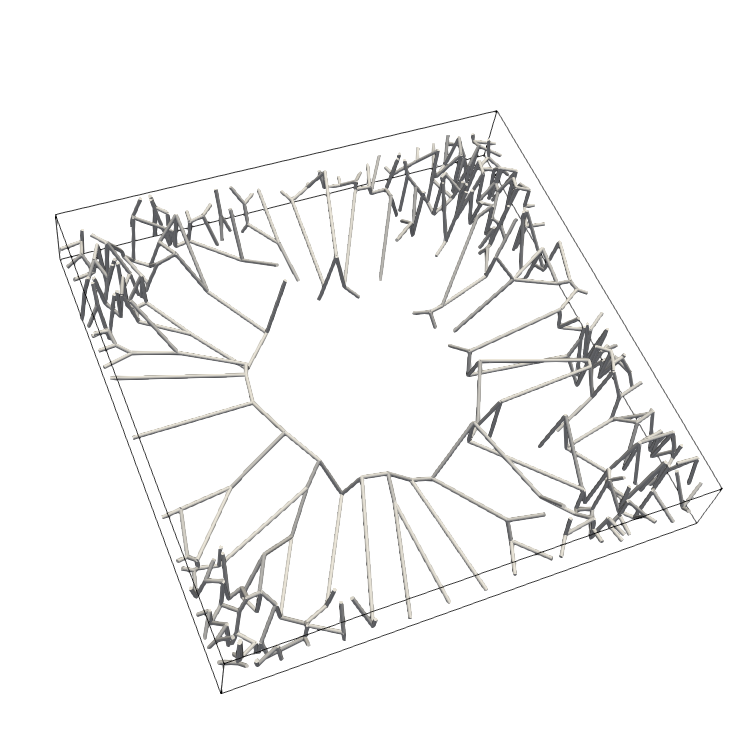}
    \caption{High vascularization, high extravascular distance}
    \end{subfigure}
    
    \caption{Examples of architectures with maximum extravascular distance increasing from left to right and higher vascular density from top to bottom.}
    \label{fig:networks}
\end{figure} 


\renewcommand{\arraystretch}{1.5}
\begin{table} [h!tb]
	\begin{center}
		\begin{tabular}{llll}
		\hline\hline
			\textbf{Symbol} & \textbf{Parameter} & \textbf{Unit} & \textbf{Range of variation} \\ \hline
			$P_{O_2}$ & \small{$O_2$ wall permeability} & $m/s$ & $0.35 \cdot 10^{-4}-3.00 \cdot 10^{-4}$
            \\
			$V_{max}$ & \small{$O_2$ consumption rate} & $\frac{mL_{O_2}}{cm^3\cdot s}$ & $0.40 \cdot 10^{-4}-2.40 \cdot 10^{-4}$
            \\
		$C_{v,in}$ & \small{$O_2$ concentration at the inlets} & $\frac{mL_{O_2}}{mL_B}$ & $2.25 \cdot 10^{-3}-3.75 \cdot 10^{-3}$
            \\
   		$\% \frac{SEEDS_{(-)}}{SEEDS_{(+)}}$ & \small{Seeds for angiogenesis} & $\%$ & $0-75$ 
        \\ 
            $S/V$ & \small{Vascular surface per unit volume} & $m^{-1}$ & $5\cdot10^3-7\cdot10^3$ \\ \hline\hline
		\end{tabular}	
            \caption{In the first three rows, the biophysical parameters of the ROM are presented along with their respective ranges of variation. The last two rows outline the hyper-parameters utilized to initialize the algorithm responsible for generating the vascular network.}
      	\label{tab:parameters}
	\end{center}
\end{table}

\subsection{Description of the ROM: a nonintrusive surrogate model based on the POD-MINN+ approach}
\label{subsec:podminn}
To construct the ROM, we exploit the POD-MINN+ approach, a nonintrusive model order reduction technique recently presented in \cite{Vitullo2024}, which the authors specifically designed to address for parametrized problems with embedded microstructure. 
%
The goal of this strategy is to derive a parameter-to-solution map using physical and geometric inputs and to generate a reconstruction of the high-fidelity oxygen concentration map in the tissue from which we calcuate the QoIs. 

Simply put, the POD-MINN+ technique consists of a projection-based regression model, incorporating neural networks and Proper Orthogonal Decomposition (POD), as in \cite{hesthaven2018non}, combined with a closure model that captures high frequency components by introducing suitable local corrections. For this purpose, the POD-MINN+ approach leverages on a specific class of neural network architectures, termed Mesh-Informed Neural Networks (MINNs) \cite{FrancoMINN}. A sketch of the methodology is shown in Figure \ref{fig:POD-MINN+}. In formulas, the whole idea can be synthesized as
\begin{equation}\label{eq:PODMINN+}
    \urom := \VV\mathcal{M}_{rb}(\boldsymbol{\mu}) + \mathcal{M}_c(\Lambda),
\end{equation}
where $\mathbb{V}\in\mathbb{R}^{N_{h}\times k}$ is the POD basis, while $\mathcal{M}_{rb}$ and $\mathcal{M}_c$ are suitable neural network architectures. The first one maps the model parameters, $\para$, onto the corresponding POD coefficients, while the second one introduces a local correction that only depends on the vascular network $\Lambda$.

In general, both $\mathcal{M}_{rb}$ and $\mathcal{M}_c$ need to process information about the vascular graph. To this end, we parametrize $\Lambda$ in terms of the extravascular distance $\dist \in \RR^{N_h}$ and the inlet characteristic function $\inlets \in \RR^{N_h}$: the former maps each point in the tissue domain to its distance from the nearest point in the vascular network; the latter, instead, assigns a unit value to the nodes located near the inlets within the computational mesh of the tissue domain. 

We avoid considering a large number of POD modes for the low-fidelity approximation of the solution manifold in the framework of complex microstructures, where the Kolmogorov $n$-width can be slow decaying. In this case, we exploit the information contained in the degrees of freedom corresponding to the extravascular distance $\dist$ and the inlet characteristic function $\inlets$. 
\\\\
Coherently with our presentation in Section \ref{sec:dlmfmc}, the POD-MINN+ approach is a model order reduction technique that leverages on a supervised learning strategy. In particular, it necessitates of some training data. The idea goes as follows. 
First, we exploit the FOM solver to construct a suitable training set $\{(\para_{i},\ufom_{i})\}_{i=1}^{n}$. As a second step, we collect all high-fidelity simulations in a matrix of snapshots and perform a Singular Value Decomposition in order to capture the most relevant modes: this results in the construction of the POD basis $\VV\in\mathbb{R}^{N_{h}\times k}$, with $k\ll N_{h}$. In general, although problems featuring complex microstructures have a slow-decay in the Kolmogorov $n$-width \cite{Vitullo2024}, we avoid considering a large number of POD modes, as the higher frequencies will be incorporated in the closure term.

Once the POD matrix has been constructed, we use it to project the solution manifold onto a linear trial subspace of dimension $k \ll N_{h}$: thanks to this maneuver, we can move our attention from FOM solutions to POD coefficients. At this point, we assemble a neural network unit responsible for the approximation of the parameter-to-POD-coefficients map. Following \cite{Vitullo2024}, we construct the latter using two MINNs, $\mathcal{M}_{rb,\eta}$ and $\mathcal{M}_{rb,d}$, and a deep feed forward neural network, $\mathcal{N}_{rb,ph}$. The first two account account for embedded 1D structure, while the last one models the effects of the physical parameters $\para_{ph}=[V_{max},C_{v,in},P_{O_2}]$. The three components act as
%
%
\begin{equation}\label{eq:POD_MINN_micro}
\mathcal{M}_{rb}(\boldsymbol{\mu}) = \Big ( \mathcal{M}_{rb,\eta}(\inlets) \odot \mathcal{M}_{rb,d}(\dist) \Big ) \mathcal{N}_{rb,ph}(\boldsymbol{\mu}_{ph}),
\end{equation}
where we recall that $\para=[\Lambda,V_{max},C_{v,in},P_{O_2}]=[\dist,\inlets,\para_{ph}].$ Here $\odot$ is the Hadamard product, modeling a suitable interaction between the inlets function $\inlets$ and the extravascular distance $\dist$.
%
Operationally, we design the three neural network architectures, in terms of width, depth, and nonlinearities, following \cite{Vitullo2024}. In particular, we rely on mesh-informed layers and hyperbolic tangent activations: for further details, we refer to \cite{Vitullo2024}.
\begin{figure}
    \centering
    \includegraphics[width=\textwidth]{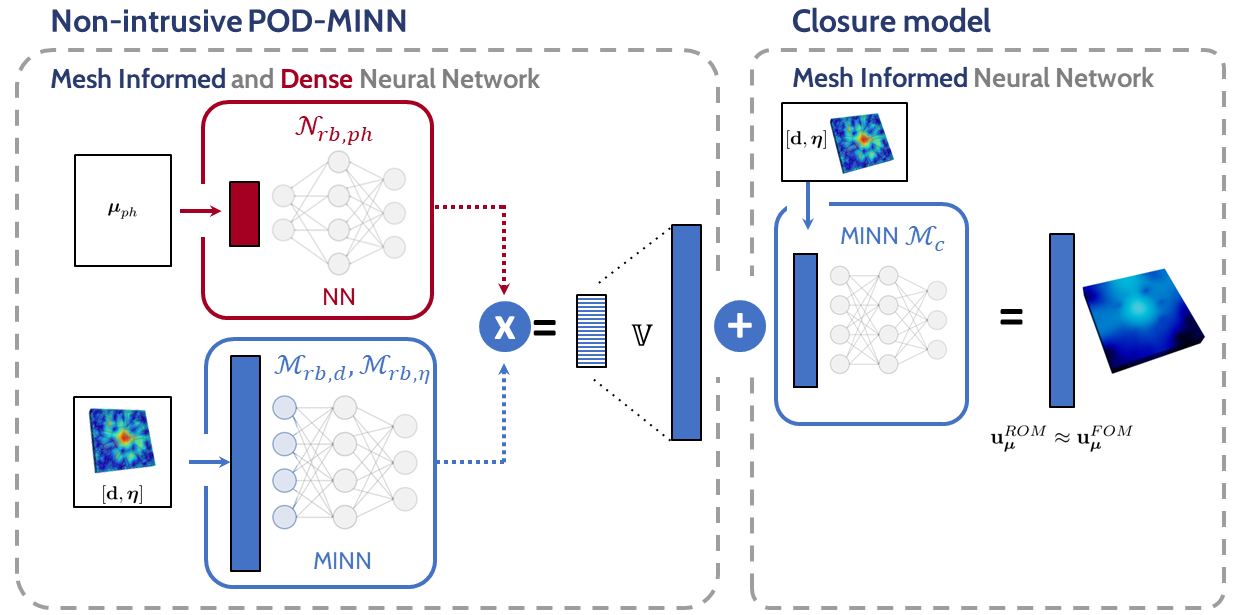}
    \caption{A sketch of the POD-MINN+ method. The macroscale parameters and the microscale ones are fed to two separate architectures, whose outputs are later combined to approximate the POD coefficients. The coefficients are then expanded over the POD basis, $\mathbb{V}$, and the ROM solution is further corrected with a closure term computed by a third network that accounts for the local features related to the high frequencies.} 
    \label{fig:POD-MINN+}
\end{figure}

We train the parameter-to-POD-coefficient network, $\mathcal{M}_{rb}$, by minimizing the loss function
\begin{equation*}
    \mathcal{E}(\mathcal{M}_{rb}) = \frac{1}{N_{train}} \sum_{i=1}^{N_{train}} \| \VV^T \ufom_i - \mathcal{M}_{rb}(\boldsymbol{\mu}_{i})\|.
\end{equation*}
To this end, we rely on the L-BFGS optimizer (default learning rate, no batching), combined with an ensamble learning strategy, i.e. by initializing and optimizing the network multiple times.  

The final step of the POD-MINN+ technique concerns the training of the closure model. Similarly to the case of POD coefficients, we process the information concerning the vascular graph using two distinct MINNs, $\mathcal{M}_{c, d}$ and $\mathcal{M}_{c,\eta}$, so that we can effectively isolate the individual impact of $\dist$ and $\inlets$. Mathematically speaking, we let
\begin{equation}
\mathcal{M}_c(\dist, \inlets) = \mathcal{M}_{c,\eta}(\inlets) \odot \mathcal{M}_{c,d}(\dist).
\end{equation}
Both architectures consist of a suitable combination of mesh-informed and dense layers: we refer to \cite{Vitullo2024} for further details about the design of $\mathcal{M}_{c, d}$ and $\mathcal{M}_{c,\eta}$.
%

In order to train the closure model, the following loss function is minimized:
\begin{align*}
\mathcal{E}(\mathcal{M}_c) = \frac{1}{N_{train}} \sum_{i=1}^{N_{train}} &\Bigg[ \regularization\| \ufom_i - \VV\mathcal{M}_{rb}(\boldsymbol{\mu}_i) - \mathcal{M}_c(\dist, \inlets) \|_{2,N_h} \\
& + (1-\regularization)\| \ufom_i - \VV\mathcal{M}_{rb}(\boldsymbol{\mu}_i) - \mathcal{M}_c(\dist, \inlets) \|_{\infty,N_h} \Bigg]
\end{align*}

Here, only the weights and biases associated with $\mathcal{M}_c$ undergo the optimization process, while those of $\mathcal{M}_{rb}$ are frozen. Notice also the hyperparameter $\regularization$, which controls the trade-off between $2$-norm and $\infty$-norm regularization. As before, we rely on the L-BFGS optimizer, combined with an ensamble learning strategy, for the minimization of the loss function.

\section{Application of the DL-MFMC method to oxygen transport and radiotherapy}
\label{sec:results}

\subsection{Computational setup for numerical simulations}

We implemented the FOM solver using a C++ based \emph{in-house} code \cite{Possenti2019b}, developed using the GetFem++ library \cite{renard:hal-02532422}. On our workstation, consisting of an AMD EPYC 7301 16-Core Processor with 2 sockets and 16 cores, each FOM simulation required roughly $w_0=25$ minutes, while the generation of a single vascular network took approximately $g=1$ second.
%
%
For the implementation of the ROM, instead, we rely on a custom Python library combining Pytorch and FEniCS. In this case, we conducted all trainings and evaluations on a Tesla V100-PCIE-32GB GPU accelerator.

\begin{table} [h!t]
	\begin{center}
		\caption{Prescribed values of input parameters in the comprehensive computational model for the high-fidelity approximation of the solution through the finite element method.}
		\begin{tabular}{c c c c c}
		\hline\hline
			\textbf{Symbol} & \textbf{Parameter} & \textbf{Unit} & \textbf{Value} & \textbf{Ref.\;\#} \\\hline
			$L$ & \small{characteristic length} & $m$ & $1 \cdot 10^{-3}$ & -- \\
			$R$ & \small{average radius} & $m$ & $4 \cdot 10^{-6}$ & \cite{Possenti2019b} \\
			$K$ & \small{tissue hydraulic conductivity} & $m^2$ & $1\cdot 10^{-18}$ & \cite{Koppl2020,Possenti2019b} \\
			$\mu_t$ & \small{interstitial fluid viscosity} & $cP$ & $1.2$ & \cite{Swartz2007} \\
			$\mu_v$ & \small{blood viscosity} & $cP$ & $3.0$ & \cite{Pries1994} \\
			$L_p$ & \small{wall hydraulic conductivity} & $m^{2}$\,s\,kg$^{-1}$ & $1 \cdot 10^{-12}$ & \cite{Possenti2019b}\\
			$\delta \pi$ & \small{oncotic pressure gradient} & $mmHg$ & 25 & \cite{Possenti2019b}  \\
			$\sigma$ & \small{reflection coefficient} & $-$ & $0.95$ & \cite{Levick1991} \\
			$D_v$ & \small{vascular diffusion coefficient} & $m^2/s$ & $2.18 \cdot 10^{-9}$ & \cite{Lucker2015H206} \\
			$N \cdot MCHC$ & \small{max. hemoglobin-bound $O_2$} & $\frac{mL_{O_2}}{mL_{RBC}} $ & $0.46$ & \cite{Secomb20041519} \\
			$\gamma$ & \small{Hill constant} & -- & $2.64$ & \cite{Welter2016,Lucker2015H206} \\
			$p_{s_{50}}$ & \small{$O_2$ at half-saturation } & $mmHg $ & $27$ & \cite{Welter2016,Moschandreou20111} \\
			$\alphaox$ & \small{$O_2$ solubility coefficient} & $\frac{mL_{O_2}/mL}{mmHg}$ &  $3.89 \cdot 10^{-5}$ & \cite{Secomb20041519} \\
			$D_t$ & \small{tissue diffusion coefficient} & $m^2/s$ & $2.41 \cdot 10^{-9}$ & \cite{Lucker2015H206} \\
			$C$ & \small{Characteristic $O_2$ concentration} & $\frac{mL_{O_2}}{mL_B}$ & $1.50 \cdot 10^{-3}$ & -- \\
	   \hline
           \hline
		\end{tabular}	
		\label{tab:parameters_fixed}
	\end{center}
\end{table}

The POD-MINN+ approach has been implemented projecting the discrete solution manifold over a trial linear subspace consisting of $k=10$ POD basis functions, as, given the diffusive nature of the problem, those were sufficient to capture the main global features of the PDE solutions. 
As discussed in Section \ref{subsec:podminn}, the training phase consisted of two distinct steps:
(i) the approximation of the POD reduced coefficients, and
(ii) the correction with the closure model.
During first phase, 
the neural networks were trained for 
at most $50$ epochs, before undergoing a lifting with respect to the original finite element space of dimension $N_h=1764$.
Then, in the second and final step, 
we trained the closure model for a total of $10$ epochs. 
The rationale behind this decision revolves around balancing the computational resources needed to construct the estimator with the intricate complexity of the neural network architecture defining the closure model. In fact, the considerable number of degrees of freedom in the closure map $\mathcal{M}_c$, stemming from the high-dimensionality of the input and output data, can pose significant challenges. Hence, it becomes essential to impose appropriate constraints on the training time to prevent over-fitting phenomena, and, consequently, a substantial degradation of the ROM as a whole.
For what concerns the loss function, 
we set the regularization parameter to $\regularization=0.75$, as that provides an acceptable fit. Refer to \cite{Vitullo2024} for more details on the empirical tests supporting this choice. 

We conclude with a final remark on the definition of the radiotherapy-related QoI that we considered for this study, that is, the Tumor Control Probability (TPC). As discussed in Section \ref{sec:qois}, we computed the TCP by leveraging on the tissue oxygen concentration map $C_t$ and the linear-quadratic model. Here, we initialized the latter using the values in Table \ref{tab:parameters_radio}, where the ratio $\alpha/\beta$ has been prescribed for a tumor tissue scenario. Finally, we set the radiation dose to $D=20 Gy$.

 
\begin{table} [h!t]
	\begin{center}
		\begin{tabular}{c c c c c}
		\hline\hline
			\textbf{Symbol} & \textbf{Parameter} & \textbf{Unit} & \textbf{Value} & \textbf{Ref.\;\#} \\\hline
		$D$ & \small{radiation dose} & $Gy$ & $20$ & -- \\
            $\alpha$ & \small{radiosensitivity parameter for 'single' hit} & $Gy^{-1}$ & $0.178$ & \cite{kelros71} \\
            $\beta$ & \small{radiosensitivity parameter  for 'multiple' hits} & $Gy^{-2}$  & $0.0455$ & \cite{kelros71} \\
            $\delta$ & \small{TIN parameter} & - & $1.38$ & \cite{Possenti2023.10.16.562646} \\
            $M$ & \small{TIN parameter} & - & $2.81$ & \cite{Possenti2023.10.16.562646} \\
            $a$ & \small{TIN parameter} & $keV/\mu m$  & $522.45$ & \cite{Possenti2023.10.16.562646} \\
            $b$ & \small{TIN parameter} &  $mmHg$ & $1.24$ & \cite{Possenti2023.10.16.562646} \\
            $N_c$ & \small{Clonogenic cells in the interstitial volume} & - & $10^8$ & \cite{delmonte2009} \\
            $LET$ & \small{Linear Energy Transfer in photons}&  $keV/\mu m$ & $2$& \cite{Durante2021} \\ \hline\hline
            \end{tabular}	
  		\caption{Input parameters values assigned in the linear-quadratic model to compute the TCP QoI.}
            \label{tab:parameters_radio}
	\end{center}
\end{table}

\subsection{Optimal management of the computational budget}

As we detailed in Section \ref{sec:optimal-policy}, the implementation of the DL-MFMC estimator entails a preliminary analysis (steps 2 - 5), necessary for the estimation of the optimal sample size $\nstar$, and the optimal sampling policy $m_0^*,m_1^*,\lambda_*.$ In particular, one of the first steps concerns the estimation of the trend coefficients $\alphareg, c_{1},c_{2},c_{3},c_{4}$, modeling the behavior of training times $t=t(n)$ and QoI correlations for varying sample size $\rho=\rho(n)$, cf. Assumptions \ref{a:trainingtime}-\ref{a:rho}. Of note, this analysis can be conducted on a small pool of FOM simulations, regardless of the computational budget $p$.

Here, we performed this preliminary analysis on a restricted pool of $n_0=300$ FOM simulations, recording the behavior of $t$ and $\rho$ for varying sample sizes, $40\le n_{1}<\dots n_{6}\le200<n_{0}$, with the $n_{j}$'s forming a uniform partition of $[40, 200]$. For each $n_{j}$ we trained the ROM multiple times as to obtain additional records for $t(n_{j})$ and $\rho(n_{j})$, with correlations being estimated over a test set of size $n_{0}-n_{j}$; then, we estimated the trend coefficients by fitting a suitable regression model. Results are in Figure \ref{fig:rho_regr} and \ref{fig:t_regr}.


\begin{figure}[htpb!]
\begin{center}
\begin{tabular}{ccc}
     \hspace{0.1\linewidth} $\overline{pO_2}$ \hspace{0.1\linewidth} 
     & \hspace{0.1\linewidth}  $\Delta pO_2$ \hspace{0.1\linewidth}  
     & \hspace{0.1\linewidth} $TCP$ \hspace{0.1\linewidth}  
\end{tabular}
\includegraphics[width=\textwidth]{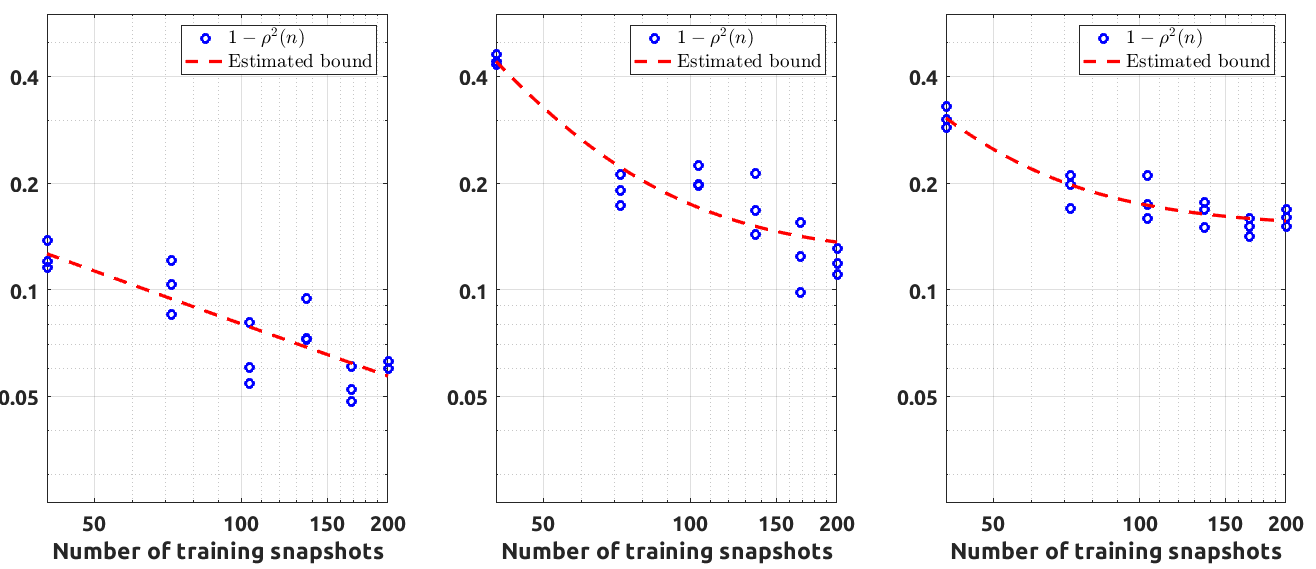}
\end{center}
\caption{Regression models for the law $1 - \rho^2(n_j) \leq c_1 n_j^{-\alphareg}+c_{2}$ for each QoI, $\overline{pO_2}$, $\Delta pO_2$, $TCP$, varying the sample sizes as $j=1,...,k$.} 
\label{fig:rho_regr}
\end{figure}


Among the three QoIs, the strongest correlation between FOM and ROM is observed for the average oxygen partial pressure, $\overline{pO_2}$. In general, this is not very surprising, in fact, deep learning-based ROMs are known to yield better approximations when it comes to smooth operators, see, e.g., \cite{FrancoROMPDE, franco2023approximation}.

On the other hand, the performances of the ROM are considerably more limited when considering 
less smooth operators, such as $\Delta pO_2$, or operators that result in nearly unimodal distributions, as the $TCP$. Further evidence of this claim lies in the estimates obtained for $c_{2}$, i.e., in the coefficient modeling the intrinsic correlation gap. In fact, we obtained higher values for
$\Delta pO_2$ and $TCP$, respectively $0.126$ and $0.151$, as opposed to the smaller estimate associated with $\overline{pO_2}$, where $c_2 = 0.003$. In particular, this goes to show that the ROM would perform better on smoother operators even when provided with an infinite amount of data.
Another interesting aspect concerns the estimated decay rate of $1 - \rho(n)^2$, namely $\alphareg$. In fact, 
although the functions settle to a higher limit value, the correlation between ROM and FOM for the variability of the partial pressure spatial and the TCP increases faster ($\alphareg \sim -2$) than the corresponding map related to the spatial average ($\alphareg \sim -1/2$).

Let us now discuss the results for the linear upper bound concerning the training time.
As shown in in Figure \ref{fig:t_regr}, the initialization of the ROM constitutes the most relevant source of computational cost. In fact, the actual computational time required for the optimization of the loss function is barely affected by the number of training samples $n$. The main reason for this lies in our design choice of limiting the number of training epochs below 50. Running the optimization process for a larger number of iterations would make the curve in Figure \ref{fig:t_regr} become steeper, with $c_{3}$ gaining dominance over $c_{4}$.
%
%

\begin{figure}[htpb!]
    \centering
    \includegraphics[width=0.65\textwidth]{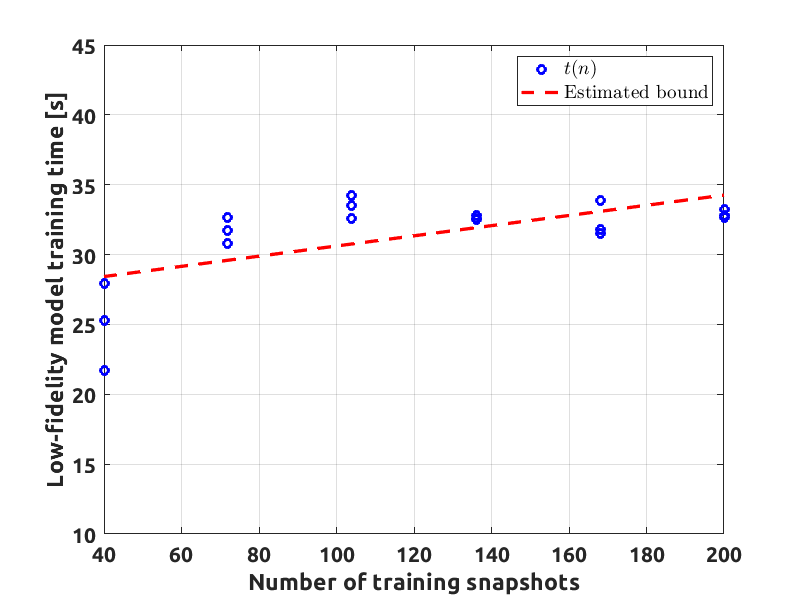}
    \caption{Regression model for the law describing the ROM training time $t(n_j) \leq c_{3} n_j+c_{4}$, varying the sample sizes as $j=1,...,k$.} 
    \label{fig:t_regr}
\end{figure}

After estimating the regression coefficients 
$\alphareg,c_1, c_2, c_3, c_4$, we focused on the retrieval of the optimal sample size $\nstar$, obtained by minimizing the upper bound in 
Lemma $\ref{lm:estimatebound}$. The results of this step, however, depend on the overall computational budget $p>0$, as that appears explicitly in the objective function to be minimized. For this reason, we postpone the discussion right below, together with the UQ analysis.

\subsection{Results of the uncertainty quantification analysis}
For the actual UQ analysis, we consider a variable computational budget of $p\in\{ 9, 12, 15, 18, 24\}$ hours.

As a first step, we leverage on the previous results to estimate the optimal sample size $n^*$, which we use to train the final ROM surrogate through the aforementioned POD-MINN+ approach. Subsequently, following steps 7-8 in the optimal policy pipeline, we estimate
the number of new high-fidelity and low-fidelity evaluations, $m_0^*$ and $m_1^*$, required for acceleration of QoI statistics and their variance reduction. Table \ref{tab:uq_samples} shows a synthetic overview of such analysis, reporting the results for three budgets of reference, namely 12, 18 and 24 hours.



\begin{table}[h!t]
    \centering
    \renewcommand{\arraystretch}{1.5} 
    \setlength{\tabcolsep}{7pt} 
    \begin{tabular}{l|lllll}
    \hline\hline
    \textbf{QoI} & \textbf{Budget} $p$  & $n^*$ & $m_0^*$  & $m_1^*$ & $\%\frac{n^*}{n^*+m_0^*}$
    \\
    \hline
    & $12\,h$& 136 & 232 & 8311 & 36.96$\%$ \\
    $\overline{pO_2}$& $18\,h$& 199 & 342 & 13522 & 36.78$\%$ \\
    & $24\,h$& 259 & 451 & 19082 & 36.48$\%$ \\
    \hline
    & $12\,h$& 126 & 266 & 6037 & 32.14$\%$\\
    $\Delta pO_2$& $18\,h$& 151 & 426  & 10053 & 26.17$\%$ \\
    & $24\,h$& 171 & 590 & 14202 & 22.47$\%$ \\
    \hline
    & $12\,h$& 98 & 292 & 6168 & 25.13 $\%$ \\
    $TCP$& $18\,h$ & 116 & 462  & 9996 & 20.07$\%$ \\
    & $24\,h$& 130 & 635 & 13914 & 16.99$\%$ \\
    \hline\hline
    \end{tabular}
    \caption{Number of FOM and ROM simulations employed for a UQ analysis of the oxygen transfer processes with the DL-MFMC estimator for three fixed reference computational budgets. In particular in the last column we report the percentage of FOM simulations required in the training phase.}
    \label{tab:uq_samples}
\end{table}

First of all, we note 
that the computational speed-up achieved through model order reduction and the short training time are effectively utilized: in fact, there is a difference of $1-2$ orders of magnitude between the number of reduced order model (ROM) and full order model (FOM) simulations. Additionally, we observe a significant distinction in the outcomes when comparing the three QoIs. Once again, the greater smoothness of the spatial-average operator, $\overline{pO_2}$, 
results in a increased reliance on the high-fidelity solver. The described trend is observed also when comparing the different allocation of FOM simulations between the training and the sampling phase, since we obtain higher percentages for $\overline{pO_2}$. As we shall see in a moment, these considerations will also be repeated in the analyses that follow.
\\\\
Given $m_0^*$ and $m_1^*$, we finally moved to the sampling phase, which allowed us to estimate the coupling coefficient $\lambda_*$ and eventually construct the DL-MFMC estimator. As we discussed in Section \ref{sec:dlmfmc}, our main interest is to provide a quantification of the uncertainties, and thus construct suitable confidence intervals based on our estimator. To appreciate the reduction in the uncertainties, we report the confidence intervals obtained by the DL-MFMC approach in comparison with those obtained via standard Monte Carlo (MC-FOM): we refer to Definition \ref{def:ic-dl-mfmc} and \ref{def:ic-mc-fom}, respectively. Hereon, the confidence level has ben set to $\gamma=99\%.$

To start, we report below the results obtained for $p=24h$, where the robustness of the DL-MFMC estimator is particularly evident.
%
\begin{align*}
\overline{pO_2}:&
\begin{cases}
    \imfmc & = 29.99 \,mmHg\, \pm 0.33 \,mmHg\, ,\\
    \ifom & = 30.09 \,mmHg\, \pm 0.58 \,mmHg\, .
\end{cases}
\\[10pt]
\Delta pO_2:&
\begin{cases}
    \imfmc & =  16.92 \,mmHg\, \pm 0.48 \,mmHg\, ,\\
    \ifom & =  17.55 \,mmHg\, \pm 0.75 \,mmHg\, .
\end{cases} 
\\[10pt]
TCP:&
\begin{cases}
    \imfmc & = 50.06 \% \pm 1.35 \% , \\
    \ifom & = 49.75 \% \pm 2.20 \% .
\end{cases}   
\end{align*}
In general, both approaches report consistent pointwise estimates. However, the DL-MFMC estimator entails a much smaller uncertainty when compared to the classical one based on standard Monte Carlo.
\\\\
Notably, the same behavior was also reported for the other computational budgets, as depicted in Figure \ref{fig:MC} and Table \ref{tab:uncertainties_ratio}.
%
In Figure \ref{fig:MC} we highlight for each QoI the trend of the point estimates $\DLMFMCestimator$ and $\FOMestimator$ with respect to the available computational budget $p$ and the contours of the uncertainties of the associated $99\%$ confidence intervals. The plots reveal various noteworthy aspects: 
\begin{itemize}
    \item[i)] the confidence intervals associated with the DL-MFMC estimator are consistently smaller than the ones obtained via standard Monte Carlo;\vspace{0.1cm}
    \item[ii)] in the DL-MFMC approach, the point estimate of $\mathbb{E}[Q(\ufom)]$ is more stable for $\overline{pO_2}$ than for $\Delta pO_2$ and $TCP$;\vspace{0.1cm}
    \item[iii)] the stability of the point estimate of $\overline{pO_2}$ is reached with a smaller budget $p$.
\end{itemize}
As before, (ii) and (iii) are easily
motivated by the additional regularity of the spatial average operator $\overline{pO_2}$, which a linear and continuous functional of $C_{t}$, as stated in Section \ref{subsec:fom}.\\\\

\begin{figure}
    \centering
    \includegraphics[width=0.49\textwidth]{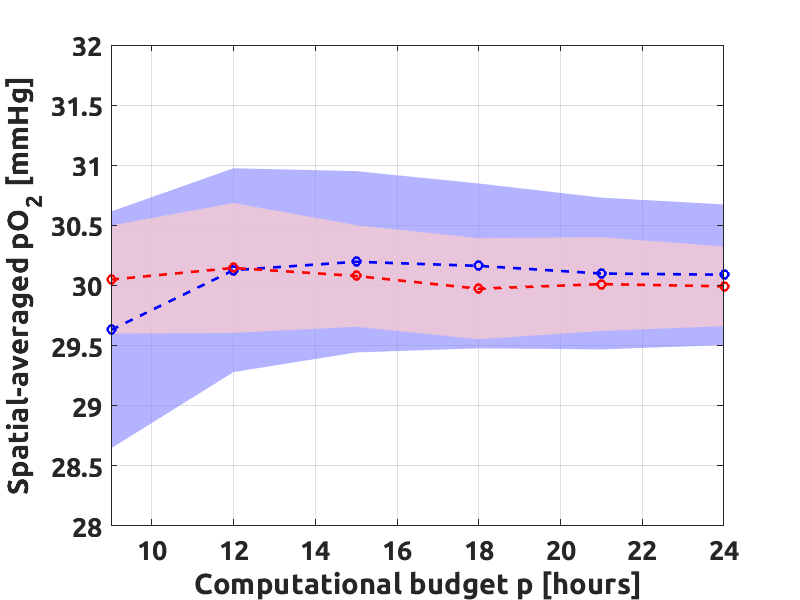}
    \includegraphics[width=0.49\textwidth]{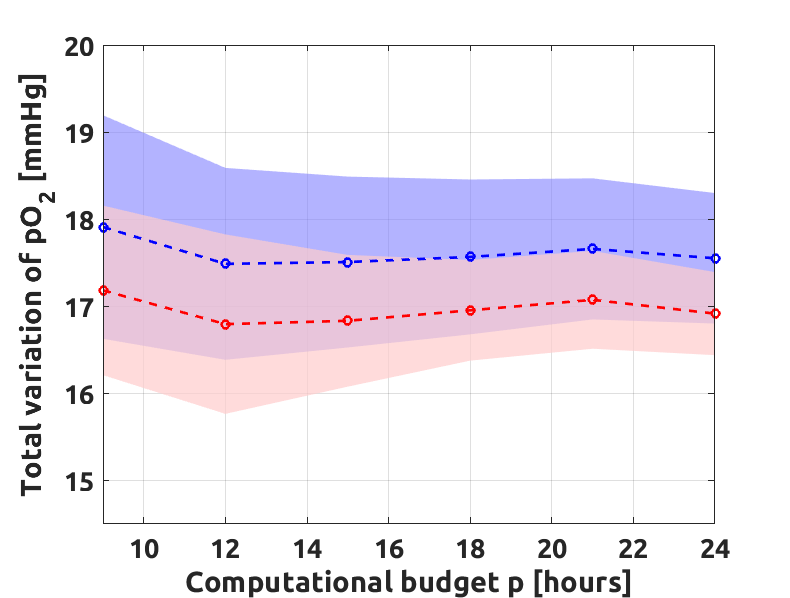}
    \\
    \includegraphics[width=0.49\textwidth]{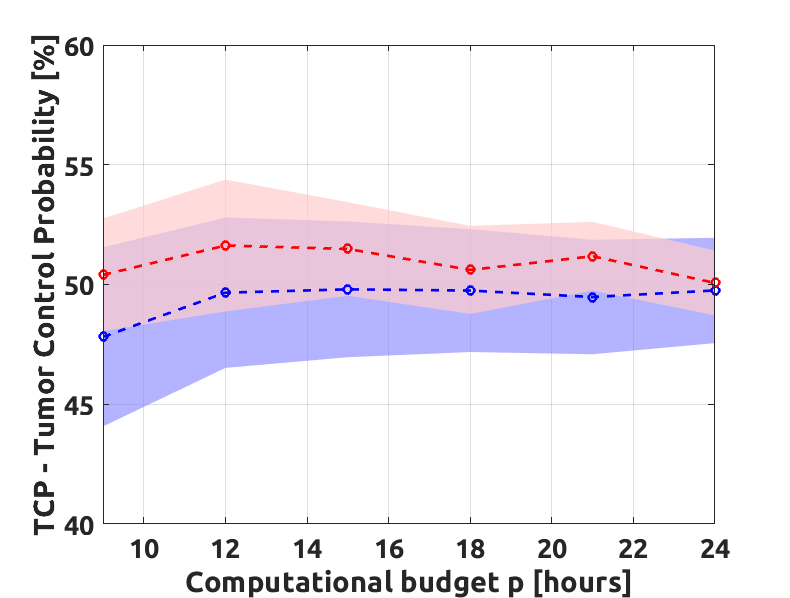}
    \\\;\\
    \includegraphics[width=0.8\textwidth]{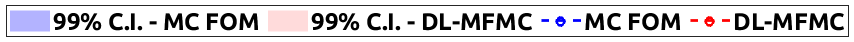}
    \caption{Confidence intervals estimates for all the QoIs, comparing the DL-MFMC estimator with the standard Monte-Carlo FOM-based, fixed $\gamma = 99\%$.} 
    \label{fig:MC}
\end{figure}

   \begin{table}
    \centering
    \renewcommand{\arraystretch}{1.5} 
    \setlength{\tabcolsep}{7pt} 
    


    \begin{tabular}{l|llll}
    \hline\hline
    \textbf{QoI} & \textbf{Budget} $p$ & \textbf{FOM simulations} & \textbf{Uncertainty} & 
    \\\hline
    & $12\,h$ & 368\;\;\;(-19.12\%) & 0.54 $mmHg$ &(-36.24\%)\\
    $\overline{pO_2}$& $18\,h$ &541\;\;\;(-20.79\%) & 0.42 $mmHg$ & (-38.68\%)\\
    & $24\,h$ &710\;\;\;(\textbf{\textcolor{ForestGreen}{-22.06\%}}) & 0.33 $mmHg$ & (\textbf{\textcolor{ForestGreen}{-43.41\%}})
    \\\hline
    & $12\,h$ &392\;\;\;(\textbf{\textcolor{red}{-13.85\%}}) & 1.03 $mmHg$ & (\textbf{\textcolor{red}{-6.50\%}})\\

    $\Delta pO_2$& $18\,h$ &577\;\;\;(-15.52\%) & 0.58 $mmHg$ & (-34.98\%)\\
    & $24\,h$ &761\;\;\;(-16.47\%) & 0.48 $mmHg$ & (-36.33\%)
    \\\hline
    & $12\,h$ & 390\;\;\;(-14.29\%) & 2.75\% & (-12.33\%)\\
    $TCP$& $18\,h$ &578\;\;\;(-15.37\%) & 1.84\% & (-28.30\%)\\
    & $24\,h$ &765\;\;\;(-16.03\%) & 1.35\% & (-38.44\%)\\     
    \hline\hline
    \end{tabular}

    \caption{Computational cost and UQ for the DL-MFMC estimator. FOM simulations = total number of high-fidelity simulations required by the computational pipeline, namely $\nstar+m_0^*$. Uncertainty = amplitude of the DL-MFMC confidence interval, Eq. \eqref{def:ic-dl-mfmc}. In parentheses, the comparison with the standard Monte Carlo estimator. E.g.: in the first row, we see that, compared to standard Monte Carlo, the DL-MFMC estimator reduced the uncertainty by 36.26\% while simultaneously requiring 19.12\% less FOM simulations.}
    %
    \label{tab:uncertainties_ratio}
\end{table}

\noindent These qualitative considerations are further confirmed by the quantitative analysis in
Table \ref{tab:uncertainties_ratio}. 
There, we focused on 
three reference budgets, $p=12, 18, 24$ hours, 
and we reported the uncertainties attained by the DL-MFMC approach (computed as the amplitude of the associated confidence interval), together with the total number of FOM simulations required by the computational pipeline (i.e., encompassing both the training and sampling phases). 
The results are reported in comparison with the MC-FOM approach, emphasizing the reduction in uncertainties and number of FOM evaluations.

In general, the advantage of the DL-MFMC approach becomes more and more pronounced for higher budgets. This is not particularly surprising if we consider that the ROM becomes more reliable when provided with a larger amount of data. In particular, as the correlation between FOM and ROM increases, the multi-fidelity approach starts to rely more and more on the surrogate model, ultimately reducing the number of FOM evaluations. Once again, this phenomenon is particularly evident for the average partial pressure, the QoI associated with the smoothest operator. However, it is interesting to note that, while the computational gain is nearly constant for $\overline{pO_2}$, the trend is much steeper for $\Delta pO_2$ and $TCP$. For instance, when increasing the computational budget, the reduction in the uncertainties for $\Delta pO_2$ goes from $-6.5\%$ to $-36.33\%$, highlight a substantial boost in the performances.
%
%

Further confirmation of this fact 
is provided in Figure \ref{fig:gain_budget}, where we compared computational resources with model uncertainties.
%
Coherently with our previous observation, we note 
that for $\overline{pO_2}$ there is a constant computational gap between the Monte Carlo and the DL-MFMC estimators. Instead, for QoIs associated to either less regular or strongly nonlinear operators, such as $\Delta pO_2$ and $TCP$, the DL-MFMC estimator becomes more advantageous when increasingly accurate estimates are required. Still, the improvement compared to a standard MC estimator remains consistent across a wide range of uncertainties. 

\begin{figure}
    \centering
    \begin{subfigure}{0.32\textwidth}
    \includegraphics[scale=0.35]{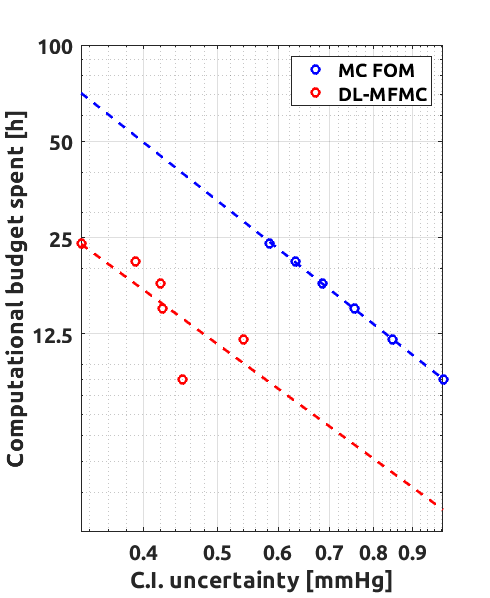}    
    \caption{$\overline{pO_2}$}
    \end{subfigure}
    \begin{subfigure}{0.32\textwidth}
    \includegraphics[scale=0.35]{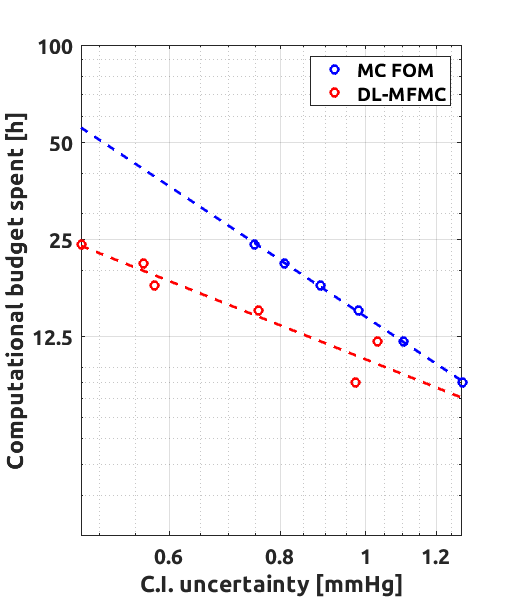}    
    \caption{$\Delta pO_2$}
    \end{subfigure}
    \begin{subfigure}{0.32\textwidth}
    \includegraphics[scale=0.35]{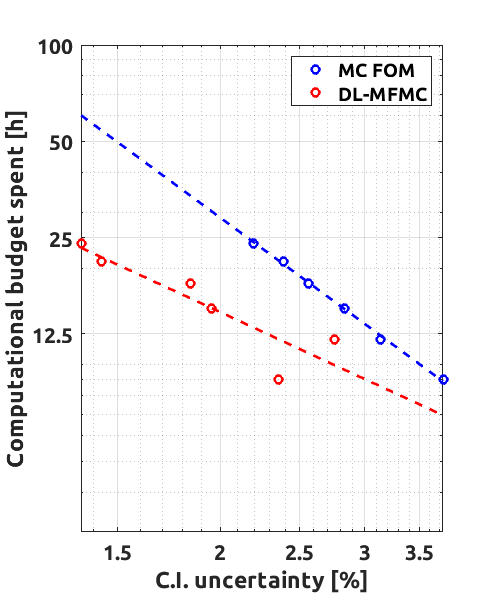}    
    \caption{$TCP$}
    \end{subfigure}
    \caption{Computational budget spent to achieve a fixed level of uncertainty for both the considered estimators, considering three different QoIs.} 
    \label{fig:gain_budget}
\end{figure}
\section{Conclusion}
\label{sec:conclusions}

This work addresses the significant computational demands associated with forward UQ for partial differential equations in multi-physics models. In particular, we propose an improvement of multi-fidelity methods \cite{peherstorfer2016optimal, peherstorfer2018survey, ng2014multifidelity}, integrating state-of-the-art deep learning techniques \cite{FrancoROMPDE,FrancoMINN} to increase the efficiency and robustness of predictions.

The primary contribution of our research lies in the development of the Deep Learning-Enhanced Multi-Fidelity Monte Carlo (DL-MFMC) method. This approach, building upon several foundational works in multi-fidelity methods and deep learning, introduces a novel synergy between full-order models (FOM) and reduced-order models (ROM), based on a supervised learning approach enhanced by deep neural networks. The result is a significant reduction in computational costs while retaining the critical features necessary for accurate modeling. By applying our DL-MFMC method to oxygen transfer in the microcirculation, using quantities of interest related to radiotherapy, we have demonstrated its ability to perform robust and reliable UQ analysis in complex, real-world scenarios. We have compared our method against traditional Monte Carlo approaches, demonstrating substantial speed-ups and increased robustness.

Looking ahead, further research should focus on refining the integration of deep learning with multi-fidelity methods, as well as improving the interpretability and scalability of these models remains a crucial area for continued development. In conclusion, our research contributes to the rapidly evolving field of computational modeling and uncertainty quantification in life sciences. The DL-MFMC approach offers a promising path forward in tackling the inherent challenges of computational expense and model accuracy in multi-scale, multi-fidelity scenarios.






\section*{Acknowledgments}
Paolo Zunino acknowledges the support of the grant MUR PRIN 2022 No. 2022WKWZA8 "Immersed methods for multiscale and multiphysics problems (IMMEDIATE)” part of the Next Generation EU program. The present research is part of the activities of project Dipartimento di Eccellenza 2023-2027, Department of Mathematics, Politecnico di Milano, funded by MUR, and of project Cal.Hub.Ria (Piano Operativo Salute, traiettoria 4), funded by MSAL. The authors are members of the Gruppo Nazionale per il Calcolo Scientifico (GNCS) of the Istituto Nazionale di Alta Matematica (INdAM).










Received xxxx 20xx; revised xxxx 20xx; early access xxxx 20xx.
\medskip

\end{document}